\documentclass[times,numbers]{amsart}
\newtheorem{pro}{Proposition}[section]
\newtheorem{teo}{Theorem}[section]

\newtheorem{lem}{Lemma}[section]

\usepackage{xifthen}
\newboolean{pedro}
\setboolean{pedro}{true}
\newcommand{\pedro}[1]{\ifthenelse{\boolean{pedro}}{\color{red!80!black}
    \setboolean{pedro}{false}\marginpar{\begin{tiny}#1\end{tiny}}}{\color{black}\setboolean{pedro}{true}}}
\newcommand{\javier}[1]{\ifthenelse{\boolean{pedro}}{\color{red}
    \setboolean{pedro}{false}\marginpar{\begin{tiny}#1\end{tiny}}}{\color{black}\setboolean{pedro}{true}}}
\newcounter{margin}

\usepackage{tikz}

\usepackage{fp}
\usetikzlibrary{math}
\theoremstyle{remark}
\newtheorem{rem}{Remark}[section]
\newtheorem{example}[teo]{Example}%[section]

\theoremstyle{definition}
\newtheorem{defi}{Definition}[section]

\begin{document}

% Enter full title and short title for running headers
\title{The Poincar\'e problem for reducible curves}
%\title{A Demonstration of the \LaTeXe\ Class File for the \textit{Oxford University Press Ltd Journal}}
%\shorttitle{A Demonstration of the \textit{OUP Journal} Class File}

% Enter the publication year and the ID number of the paper
%\volumeyear{2009}
%\paperID{rnn999}

% Author name(s)
\author{P. Fortuny Ayuso}
\address{Departamento de Matem\'{a}ticas, Universidad de Oviedo, Spain.}
\email{fortunypedro@uniovi.es}
 
\author{J. Rib\'{o}n}
\address{Instituto de Matemática e Estatística, Universidade Federal Fluminense, Brazil.}
\email{jribon@id.uff.br}

%\affil{1}
%\affil{2}}
% Abbreviated author name for running headers
%\abbrevauthor{P. Fortuny Ayuso and J. RibÃ³n}
% Abbreviated author name for first page header
%\headabbrevauthor{Author, F., and S. Author}

\subjclass[2020]{Primary 32S05; Secondary 32S65, 14H20}

%\address{
%Departamento de Matem\'{a}ticas, Universidad de Oviedo, Spain.\\
%and\\
%Departamento de An\'{a}lise, Universidade Federal Fluminense, Brazil.}

% Address / e-mail address of corresponding author
%\correspdetails{fortunypedro@uniovi.es}

\providecommand{\keywords}[1]{\noindent\textbf{\textit{Keywords:}} #1.}

% Received/revised/accepted dates will be entered by the publisher during production of an accepted paper. Please do not edit these placeholders for submission.
%\received{1 Month 20XX}
%\revised{11 Month 20XX}
%\accepted{21 Month 20XX}

% Enter details of editor communicating this article
%\communicated{A. Editor}

\begin{abstract}
We provide sharp lower bounds for the multiplicity of a local holomorphic foliation defined in a 
complex surface in terms of data associated to a germ of invariant curve.
Then we apply our methods to invariant curves whose branches are isolated, i.e. 
they are never contained in  non-trivial analytic  families of equisingular invariant curves.
In this case we show that the multiplicity of an invariant curve is at most twice the 
multiplicity of the foliation. Finally, we apply the local methods to foliations in 
the complex projective plane.
\end{abstract}

\maketitle
{\it Keywords}: 
Isolated Invariant Curve, Singularity of Holomorphic Foliation, Multiplicity, 
Poincar\'{e} Problem.

\section{Introduction}
The Poincar\'{e} problem (bounding the degree, or equivalently, the genus of an invariant curve of a foliation in projective space) has been thoroughly studied lately 
\cite{Car:1994, Brunella:1997, LN:2002, Pereira:2002, Cavalier-Lehmann:2006, Galindo-Monserrat:2014,Genzmer-Mol:2018} (to cite several relevant instances). 
  We want to obtain lower bounds for the complexity of a foliation 
in terms of data associated to an invariant curve and as much as possible not on the foliation itself, following an approach that  
is similar, in spirit, to the point of view of Cerveau and Lins-Neto in \cite{CL:1991}.
Indeed, one of the main contributions of the paper is that its methods do not depend on the reduction of singularities of the foliation and moreover, 
some of its results do not depend on the foliation and depend just on the invariant curve.

In a previous work \cite{Cano-Fortuny-Ribon:2020} with J. Cano,
we covered the local case for irreducible \emph{branches} (local analytic curves with a single irreducible component). There we defined the concept of \emph{virtual multiplicity} of an analytic branch $\gamma$: if $n$ is its multiplicity and $p_1/q_1,\ldots, p_g/q_g$ are its characteristic exponents, then $\mu(\gamma)$ is the denominator of the last-but-one characteristic exponent: $\mu(\gamma)=q_{g-1}$. Despite its seemingly artificial nature, it has an intuitive geometric interpretation: let $\tau:(\mathbb{C}^2,0)\rightarrow (\mathbb{C}^{2},0)$ be the ramification map $\tau(u,y)=(u^n,y)$, and let $\pi_{\tau}$ be the minimal resolution of singularities of $\tau^{-1}(\gamma)$ (which is a tree, as $\tau^{-1}(\gamma)$ is a union of $n$ non-singular branches). Then $\mu(\gamma)$ is exactly the number of irreducible components of the exceptional divisor $E_{\tau}=\pi_{\tau}^{-1}(0,0)$ which contain no center of the sequence $\pi_{\tau}$, or equivalently components with $-1$ self-intersection  (so to say, \emph{terminal} components).

In the case of a reducible curve $\Gamma$,  we generalize the virtual multiplicity in two ways. The first one is the obvious one: if $\tau:(\mathbb{C}^2,0)\rightarrow (\mathbb{C}^2,0)$ is a ramification map which turns $\Gamma$ into a union of smooth branches and $\pi_{\tau}$ is the resolution of singularities of $\tau^{-1}(\Gamma)$, then one can define the \emph{terminal virtual multiplicity} $\mu_T(\Gamma)$ as the number of terminal components in $E_{\tau}=\pi_{\tau}^{-1}(0,0)$. However (see Example \ref{ex:terminal-and-divisorial}), this number might be too low, and one can also consider the set of irreducible components of $E_{\tau}$ meeting one of the branches of the strict transform of $\tau^{-1}(\Gamma)$ by $\pi_{\tau}$: this number we call the \emph{divisorial virtual multiplicity} $\mu_D(\Gamma)$, which is at least equal to $\mu_T(\Gamma)$. If $\Gamma$ is irreducible, $\mu_D(\Gamma)=\mu_T(\Gamma)$ (and both equal $\mu(\Gamma)$) but in the general case they may differ. 
  We denote by $\nu_0 ({\mathcal F})$ and $\nu_{0} (\Gamma)$ the multiplicities at the origin of ${\mathcal F}$ and $\Gamma$ respectively.  
Our first result is:

\begin{teo}
\label{teo:genr}
  Let ${\mathcal F}$ be a germ of holomorphic foliation defined in a neighborhood of $0$
  in ${\mathbb C}^{2}$. Let $\Gamma$ be a germ of singular invariant curve. Then we have 
  \[
    \nu_0 ({\mathcal F})  \geq
    {\max \left(\mu_{T} (\Gamma), \frac{\mu_{D}(\Gamma)}{2} \right)} .
  \]
\end{teo}
Notice how only the geometric structure of the invariant curve $\Gamma$ is relevant: there is no hypothesis on $\mathcal{F}$. We provide examples showing that the bound in Theorem \ref{teo:genr} is sharp.  Moreover, the values of $\mu_{T} (\Gamma)$ and $\mu_{D}(\Gamma)$ can be calculated directly from the Puiseux expansions of 
the irreducible components of $\Gamma$.  

The irreducible components of $\tau^{-1} (\Gamma)$ can be partitioned in packages, where each package contains those components 
of $\tau^{-1} (\Gamma)$  whose strict transform intersects the same component of $E_{\tau}$.
If we consider a curve $\Gamma' \subset \tau^{-1} (\Gamma)$ that contains exactly one curve in each package, we obtain $\mu_{D}(\Gamma) = \nu_{0} (\Gamma')$.
So we obtain a linear lower bound  $\nu_{0} ({\mathcal F}) \geq \nu_{0} (\Gamma')/ 2$ for the multiplicity of ${\mathcal F}$ in terms of the multiplicity of a subcurve 
$\Gamma'$ of $\tau^{-1} (\Gamma)$. Moreover, since $\nu_{0} (\Gamma) = \nu_{0} (\tau^{-1} (\Gamma))$,
we get $\nu_{0} ({\mathcal F}) \geq \nu_{0} (\Gamma)/ 2$ if there is exactly one curve in each package.

Once the most general case is studied, we turn our sight to curves
$\Gamma$ from whose multiplicity $\nu_{0} (\Gamma)$ one can compute a
non-trivial lower bound of $\nu_{0} ({\mathcal F})$. It is here that
$\Gamma$ needs to be related to $\mathcal{F}$: following the ideas by
Corral and Fern\'{a}ndez-S\'{a}nchez in \cite{Corral-Fernandez:2006},
we study the case where all the branches $\gamma$ of $\Gamma$ are
isolated invariant curves of $\mathcal{F}$
(cf. \cite{Camacho-Lins-Sad:1984}): separatrices $\gamma$ which do not belong to a non-constant one-dimensional analytic family of
equisingular curves invariant for $\mathcal{F}$.
 
To tackle this problem we introduce the less stringent notion of {\it
  weak isolation} for invariant curves (cf. Definition
\ref{def:weak_isolation}), which includes both the isolated case
treated in \cite{Corral-Fernandez:2006}, and the case where $\Gamma$
has nodal singularities, treated in \cite{CL:1991}. The value of this
generalization is that weak isolation is invariant by blow-ups
(Proposition \ref{pro:weak_invariance}) and it only rules out very
specific families of equisingular invariant curves determined by
$\Gamma$. We obtain an analogue of Theorem \ref{teo:genr} for the
weakly isolated case.
 \begin{teo}\label{teo:poincare-reducible-intro}
  Let $\Gamma$ be a singular curve that is invariant by a germ of holomorphic
  foliation $\mathcal{F}$ defined in a neighbourhood of the origin in
  $\mathbb{C}^2$. 
 Assume that   $\Gamma$ is weakly isolated.
 %integral curves of $\mathcal{F}$. 
 Then $2\nu_0(\mathcal{F})\geq \nu_0(\Gamma)$.
 \end{teo}
 In this case, the bound in Theorem \ref{teo:genr} is improved dramatically 
 since we do not need to remove any irreducible component of $\Gamma$.

Then we move on to addressing the global Poincar\'{e} problem. In order to do this, lower bounds for the vanishing number 
$Z_{P} ({\mathcal F}, \gamma)$ along a branch $\gamma$ (equal to the GSV-index \cite{zbMATH04196386:1992} except in the singular case) are required. 
We prove the following result relating the vanishing number along $\gamma$ of $\mathcal{F}$ to that of $df$ where $f=0$ is a reduced equation of $\Gamma$:
\begin{teo}
\label{teo:half}
Let ${\mathcal F}$ be a germ of foliation 
defined in a neighborhood of  a point $P$ in a complex surface.
Let $\Gamma$ be a germ of weakly isolated ${\mathcal F}$-invariant singular curve  
in a neighborhood of $P$
of reduced equation $f=0$. 
Denote by ${\mathcal H}$ the foliation $df=0$.
Consider a branch $\gamma$ of $\Gamma$ at $P$. 
Then we have $Z_{P} ({\mathcal F}, \gamma) \geq  Z_{P} ({\mathcal H}, \gamma)/2$.
\end{teo} 
Theorem \ref{teo:half} is the analogue of Theorem \ref{teo:poincare-reducible-intro} for
the vanishing number along a germ of invariant curve. 

Our results conclude with the following application of the previous ideas to the global Poincar\'{e} problem for 
holomorphic foliations in the complex projective plane ${\bf C}{\bf P}(2)$. 
\begin{teo}
\label{teo:main}
Let $\Gamma$ be an algebraic curve that is invariant by a foliation ${\mathcal F}$ of 
${\bf C}{\bf P}(2)$.
Suppose that all singular points $P$ of $\Gamma$ satisfy that 
 the germ of $\Gamma$ at $P$ is weakly isolated.
%\begin{itemize}
%\item $P$ is a normal crossings singularity of $\Gamma$ or
%\item the germ of $\Gamma$ at $P$ is isolated.
%\end{itemize}
Then $\deg (\Gamma) \leq 2 \deg ({\mathcal F}) +2$. Moreover, 
$\deg (\Gamma) \leq 2 \deg ({\mathcal F}) +1$ holds if $\Gamma$ is irreducible.
\end{teo}
Thus, improving the results of \cite{Corral-Fernandez:2006},
we prove that the degree of $\Gamma$ admits a linear bound of slope $2$ in terms of the degree of $\mathcal{F}$.

We are convinced that the bound in Theorem \ref{teo:main} is optimal in the slope. It may be possible to decrease the intercept $2$ to a lower value but we have not found any examples.

Finally, one of the most relevant properties of Theorems \ref{teo:genr} , \ref{teo:poincare-reducible-intro},  \ref{teo:half} and  \ref{teo:main}
is that they do not depend at all on the reduction of singularities of $\mathcal{F}$ or on the relation between the desingularization of $\Gamma$ 
and the pull-back of $\mathcal{F}$ to it. The only hypothesis relating $\Gamma$ to $\mathcal{F}$ is that the former is invariant by the latter (and the 
weak isolation properties in Theorems \ref{teo:poincare-reducible-intro},
 \ref{teo:half} and  \ref{teo:main}). 
 %Moreover, we can replace   weak  
 %isolation by a much weaker property, we just need to rule out specific families of equisingular invariant curves 
 %determined by $\Gamma$ (Theorem \ref{teo:poincare-reducible}  and Remark \ref{rem:weak_isolation}).

\section{Setting}\label{sec:setting}
In this section we introduce the invariants and main formulas that we are going to use in
order to obtain lower bounds for the multiplicity of a foliation in
terms of an invariant curve.
\begin{defi}
  Let $f \in {\mathbb C} \{x,y \} \setminus \{0\}$. We define the multiplicity (or
  vanishing order) $\nu_0 (f)$ of $f$ at $0 \in {\mathbb C}^{2}$ as the unique
  $k \in {\mathbb N}$ such that $f \in {\mathfrak m}^{k} \setminus {\mathfrak m}^{k+1}$
  where ${\mathfrak m}$ is the maximal ideal of the local ring ${\mathbb C}\{x,y\}$.  We
  define $\nu_0 (0)= \infty$.

  Let $\Gamma$ be a germ of reduced complex analytic curve defined in a neighborhood of
  $0$ in ${\mathbb C}^{2}$.  It is given by a reduced equation $f=0$ where
  $f \in {\mathbb C} \{x,y\}$. We define the multiplicity $\nu_0 (\Gamma)$ of $\Gamma$ at
  $0 \in {\mathbb C}^{2}$ as $\nu_0 (\Gamma) = \nu_0 (f)$.
%as the unique $k \in {\mathbb N}$ such that 
%$f \in {\mathfrak m}^{k} \setminus {\mathfrak m}^{k+1}$ where ${\mathfrak m}$ is the
%maximal ideal of the local ring ${\mathbb C}\{x,y\}$.
\end{defi}
\begin{rem}
  Let $\Gamma = \gamma^1 \cup \hdots \cup \gamma^n$ be the decomposition of $\Gamma$ in
  irreducible components. We have
  $\nu_0 (\Gamma)= \nu_0 (\gamma^1) + \hdots + \nu_0 (\gamma^n)$.
\end{rem}
Next, we define the multiplicity of a foliation.
\begin{defi}
\label{def:mul}
Let ${\mathcal F}$ be a germ of holomorphic foliation defined in a neighborhood of $0$ in
${\mathbb C}^{2}$.  Let $X = a(x,y) \partial / \partial_x + b(x,y) \partial / \partial_y$
be a holomorphic vector field inducing the foliation ${\mathcal F}$ and such that
$\mathrm{Sing} (X) \subset \{0\}$.  We define
$\nu_0 ({\mathcal F}) = \min (\nu_0 (a), \nu_0 (b))$.
%as the unique $k \in {\mathbb N}$ such that 
%$\{ a,b\}$ is contained in ${\mathfrak m}^{k}$ but it is not contained in ${\mathfrak m}^{k+1}$.
\end{defi}
Hertling's formula \cite{Hertling:2000} (see equation \eqref{equ:Hertling} below) relates
$\nu_0 ({\mathcal F})$ to indices associated to a sequence of blow-ups. The following
definitions cover all the necessary concepts to state it.
 \begin{defi}
\label{def:zero}
Consider the setting in Definition \ref{def:mul}.
%Let ${\mathcal F}$ be a germ of holomorphic foliation defined in a neighborhood of 
%$0$ in ${\mathbb C}^{2}$. It is induced by a germ of holomorphic 
%vector field $X= a(x,y) \partial / \partial_x + b(x,y) \partial / \partial_y$
%such that $\mathrm{cod} (\mathrm{Sing} (X)) \geq 2$. 
Let  $\gamma$ be a germ of irreducible invariant curve for $\mathcal{F}$
%in a neighborhood of $P$ 
and consider 
a Puiseux parametrization $\alpha$
of $\gamma$.  We define $Z_{P} ({\mathcal F}, \gamma)$ as the vanishing order of 
$\alpha^{*} X$ at the origin.
\end{defi}
\begin{defi}
\label{def:tang}
Consider the setting in Definition \ref{def:mul}.
%Let ${\mathcal F}$ be a germ of holomorphic foliation defined in a neighborhood of 
%$0$ in ${\mathbb C}^{2}$ and induced by a germ of holomorphic 
%vector field $X= a(x,y) \partial / \partial_x + b(x,y) \partial / \partial_y$
%with $\mathrm{cod} (\mathrm{Sing} (X)) \geq 2$. 
Let $\gamma$ be an irreducible germ of curve defined of irreducible
equation $f=0$ where $f \in {\mathbb C} \{x,y\}$. We define the tangency
order between ${\mathcal F}$ and $\gamma$ at $0$ as
\[ \mathrm{tang}_0 ({\mathcal F}, \gamma) = \dim_{\mathbb C}
\frac{{\mathbb C}\{x,y\}}{ (f, X(f))} . \]
\end{defi}
\begin{rem}
  The indices defined in Definitions \ref{def:zero} and \ref{def:tang} are
  invariant under change of coordinates and hence they can be defined at
  any point of a smooth complex surface.
\end{rem}
\begin{rem}
\label{rem:zero_index}
Notice that $Z_{0} ({\mathcal F}, \gamma) \geq 0$ and 
$\mathrm{tang}_0 ({\mathcal F}, \gamma) \geq 0$.
Moreover $Z_{0} ({\mathcal F}, \gamma) = 0$ is equivalent to 
$0 \not \in \mathrm{Sing}({\mathcal F})$. 
Moreover, $\mathrm{tang}_0 ({\mathcal F}, \gamma) = 0$ holds if and only if 
$0 \not \in \mathrm{Sing}({\mathcal F})$ and $X$ is transverse to $\gamma$ at $0$.
\end{rem}
\begin{rem}
The index $Z_{0} ({\mathcal F}, \gamma)$ coincides with the GSV index 
(G\'{o}mez Mont-Seade-Verjovsky) if $\gamma$ is smooth
  \cite{zbMATH04196386:1992}.  
\end{rem}
  Let $(M,P_0)$ be a germ of complex
analytic surface, and $\pi=\pi_1\circ\cdots\circ \pi_k$ be a sequence of
blow-ups where $\pi_1$ is the blow-up of $P_0$ and, for $1\leq l \leq k$,
$\pi_{l}$ is the blow-up of a point $P_{l-1}$ in
$(\pi_{1}\circ\cdots\circ \pi_{l-1})^{-1}(P_0)$. For
$1\leq l \leq k$, we shall denote $\tilde{\pi}_l=\pi_1\circ\cdots\circ \pi_l$ the
composition, and $E_l=\tilde{\pi}^{-1}(P_0)$, $E=\pi^{-1}(P_0)$ and
$D_l=\pi^{-1}_l(P_{l-1})$. Abusing notation, we shall also call $D_l$    the
strict transform of $D_l$ by $\tilde{\pi}_{l+1}, \hdots, \tilde{\pi}_{k-1}$ and  
the whole blow-up process $\pi$. The
following notion is just a matter of brevity: a \emph{trace point} of
   $D_l \subset E_j$ is a non-singular point of $E_{j}$ belonging to $D_l$. A point in
$E_j$ which is not a trace point will be called a \emph{corner} (of either
$E_j$ or $D_l \subset E_j$).  

Given a germ of complex foliation $\mathcal{F}$ in $(M,P_0)$, and a germ of analytic curve $\Gamma$ at $P_0$, we denote by $\mathcal{F}_l$ and $\Gamma_l$, respectively, their strict transforms by $\tilde{\pi}_l$, setting $\mathcal{F}_0=\mathcal{F}$ and $\Gamma_0=\Gamma$ for completeness.
\begin{defi}
\label{def:not}
We denote by $\mathrm{Inv} (E)$ the union of the irreducible
components of $E$   that are   invariant for the foliation $\mathcal{F}_k$. An
irreducible component of $E$ in $\mathrm{Inv}(E)$ will be called an
\emph{invariant component}, whereas one not in $\mathrm{Inv}(E)$ will
be called a \emph{dicritical} component.
\end{defi}
\begin{defi} 
%We denote by $\mathrm{Inv} (E)$ the union of the invariant irreducible components of $E$.
  The set of connected components of $\mathrm{Inv}(E)$ will be denoted
  ${\mathcal I}$. An element $H\in {\mathcal I}$ will be interpreted
  (without confusion) as such a connected component, or as a set whose
  elements are the irreducible components of $E$ contained in $H$. The set $\mathcal{I}$ is empty if 
    $E$ has no invariant irreducible components.  
 % all the divisors are non-invariant.
\end{defi}
\begin{defi}
  Given an irreducible component $D_j$ of $E$ ($1 \leq j \leq k$), its
  \emph{weight} $w (D_j)$ is the multiplicity of any germ of analytic
  branch $\gamma$ such that its strict transform $\gamma_j$ is smooth and
  intersects transversally $D_j$ at a trace point.
\end{defi}
\begin{rem}
  It is easy to see that if $P_j$ is a trace point of some
  $D_l \subset E_j$, we have $w(D_{j+1})=w(D_l)$. On the other hand if
  $P_j$ is a corner point belonging to    irreducible components $D_l$ and $D_{l'}$
 of    $E_j$, we have $w(D_{j+1})=w(D_l) + w(D_{l'})$.
\end{rem}
% \begin{tiny}
%   In what remains of the paper, the term \emph{divisor} refers exclusively to one of the irreducible components of the exceptional divisor, for the sake of simplicity. Thus, \emph{non-invariant} (resp. \emph{invariant}) \emph{divisor} means a \emph{non-invariant} (resp. \emph{invariant}) irreducible component of the exceptional divisor.
%   \end{tiny}
\begin{defi}\label{def:non-dicritical-valence}
  Let $D_j$ be a dicritical component of $E$. The \emph{non-dicritical
    valence} $v_{\overline{d}} (D_j)$ is the number of invariant components
  $D_l$   of $E$   such that $D_j \cap D_l \neq \emptyset$.
\end{defi}
\begin{defi}
  Given an irreducible component $D_j\subset E$, and $P\in D_{j}$, we
  define
%  Let ${\mathcal F}$ be a germ of foliation defined in a neighborhood of
%  $P_0$.  Given $P \in D_j \subset E$, we define
\begin{itemize}
\item
  $\kappa_{P} ({\mathcal F}_k, D_j) = \mathrm{tang}_{P} ({\mathcal F}_k,
  D_j)$ if $D_j\not\in\mathrm{Inv}(E)$. Otherwise:
\item $\kappa_{P} ({\mathcal F}_k, D_j) = Z_{P} ({\mathcal F}_k, D_j) -1$
  if $P$ is a corner point of $E$ and both   irreducible components 
  $D_j$ and $D_l$ of $E$  
  containing $P$ are invariant, or finally
\item $\kappa_{P} ({\mathcal F}_k, D_j) = Z_{P} ({\mathcal F}_k, D_j) $ if
  $D_j$ is invariant but we are not in the preceeding case.
\end{itemize}
\end{defi}
\begin{rem}
  The index $\kappa_{P} ({\mathcal F}_k, D_j)$ is non-negative,  and it is zero only when 
    \begin{itemize} 
  \item $P$ is a regular point
  of $\mathcal{F}_k$ and the separatrix of $\mathcal{F}_k$ through $P$ is either $D_j$ or 
  transverse to $D_j$ or 
  \item $P$ is a corner point, the germ of $E$ at $P$ is invariant and $ Z_{P} ({\mathcal F}_k, D_j) =1$.
  \end{itemize}
\end{rem}
Finally, as we shall use this concept frequently, we say, in general, that
a germ of foliation at $(M,P_0)$ is \emph{$1$-dicritical} if the
exceptional divisor $D_1$ is non-invariant.

The initial formulas relating the multiplicity $\nu_{P_0} ({\mathcal F})$ to vanishing or 
tangency indexes after blow-up are
\begin{equation}
\label{equ:non-dic}
 \nu_{P_0} ({\mathcal F}) = \sum_{P \in   \pi_1^{-1} (P_{0}) } Z_{P} ({\mathcal F}_1, D_1) - 1 
\end{equation}   
if ${\mathcal F}$ is non-$1$-dicritical at $P_0$ and  
\begin{equation}
\label{equ:dic}
\nu_{P_0} ({\mathcal F}) = \sum_{P \in  \pi_1^{-1} (P_{0})} \mathrm{tang}_{P} ({\mathcal F}_1, D_1) + 1 
\end{equation}  
otherwise. Equation (\ref{equ:non-dic}) was generalized by Camacho, Lins Neto and Sad 
for the case where ${\mathcal F}_k$ is non-dicritical, i.e. $D_j$ is invariant for any 
$1 \leq j \leq k$, in \cite{Camacho-Lins-Sad:1984}.
The general formula that holds for every situation was discovered by Hertling
\cite{Hertling:2000}: 
 \begin{equation}
 \label{equ:Hertling}
\nu_{P_0} ({\mathcal F})  +1 = \sum_{D_j} \sum_{P \in D_j} w(D_j) \kappa_P ({\mathcal F}_k, D_j)  
 +  \sum_{D_j  \not \subset \mathrm{Inv}(E)} w(D_j)  (2 - v_{\overline{d}} (D_j)) .
\end{equation}
\begin{rem}\label{rem:hidden} In Hertling's formula, for any $H\in \mathcal{I}$, we have
 \begin{equation}
 \label{equ:rem_contribh}
    \sum_{D_j \in H} \sum_{P \in D_j} w(D_j) \kappa_P ({\mathcal F}_k, D_j) \geq \min_{D_j \in H} w(D_j) 
\end{equation}
by \cite[Proposition 3.7]{Cano-Fortuny-Ribon:2020}.  This will be one
of the main tools in our approach since we can detect ``hidden" index
contributions associated to components $H$ in ${\mathcal I}$ whose
intersection with the strict transform of an invariant curve is
empty. Moreover, the previous inequality is extremely useful, as we do
not need to require that ${\mathcal F}_k$ is a reduction of
singularities of ${\mathcal F}$. It is one of the reasons why we
do not need a desingularization of ${\mathcal F}$ in our arguments.
\end{rem}  
We are also interested in how the vanishing order $Z_{P_0} ({\mathcal F}, \gamma)$ behaves
under blow-up when $\gamma$ is an irreducible germ of invariant curve. Indeed, 
if $\{ P_1 \} = \gamma_1 \cap D_1$, we  have 
\begin{equation}
\label{equ:znon-dic}
 Z_{P_0} ({\mathcal F}, \gamma) = Z_{P_1} ({\mathcal F}_1, \gamma_1) + 
\nu_{P_0} (\gamma) (\nu_{P_0} ({\mathcal F}) -1) 
\end{equation}
if ${\mathcal F}$ is non-$1$-dicritical at $P_0$ and 
\begin{equation}
\label{equ:zdic}
 Z_{P_0} ({\mathcal F}, \gamma) = Z_{P_1} ({\mathcal F}_1, \gamma_1) + 
\nu_{P_0} (\gamma) \nu_{P_0} ({\mathcal F}) 
\end{equation}
otherwise (cf. \cite[Proposition 14.26]{Ilya-Yako:2008}).
%This formula can be found in the 
%Ilyashenko-Yakovenko book (indeed only formula \ref{equ:znon-dic}, but the proof 
%of formula \ref{equ:zdic} is analogous).  \marginpar{reference}
\section{Bounds ``up to the last Puiseux exponent''}  
\label{sec:last}
Let $\Gamma$ be a (possibly reduced) germ of irreducible analytic
curve, invariant by a germ of holomorphic foliation ${\mathcal F}$
defined in a neighborhood of the origin of ${\mathbb C}^{2}$.  We want
to obtain a lower bound for the multiplicity of ${\mathcal F}$ in
terms of data associated to $\Gamma$, without imposing any additional
condition on $\mathcal{F}$, generalizing the results proved for the irreducible case in \cite{Cano-Fortuny-Ribon:2020}. In this
section, we find a lower bound obtained by, roughly speaking,
discarding the contribution to the multiplicity of the curve $\Gamma$
provided by the last Puiseux characteristic exponents of its
irreducible branches.

Let us fix the notation for this section. We assume that both
$\mathcal{F}$ and $\Gamma$ are singular at $(0,0)\in \mathbb{C}^2$
(i.e. $\Gamma$ has multiplicity at least $2$ and $\mathcal{F}$ at least
$1$). If $g \geq 1$ is the genus of $\Gamma$ (i.e. the number of Puiseux
characteristics) and $p_1/q_1, \hdots, p_{g}/q_{g}$ are the Puiseux
characteristic exponents, we defined in \cite{Cano-Fortuny-Ribon:2020} the
virtual multiplicity $\mu(\Gamma)$ as $q_{g-1}$. Moreover, we proved that
\begin{equation}
\label{equ:in_virtual}
\nu_{0} ({\mathcal F}) \geq \mu (\Gamma) . 
\end{equation}
and showed that the inequality is sharp. As a consequence, in order to
obtain a sharp lower bound of $\nu_{0} ({\mathcal F})$ in terms of
$\Gamma$ we need to discard the ``contribution of the last Puiseux
exponent".  %Moreover, the virtual multiplicity can be computed in terms of the Puiseux parametrization of $\Gamma$ as we will see below.

In the general case, decompose $\Gamma$ into its irreducible components
$\Gamma=\gamma^1 \cup \cdots \cup \gamma^q$. 
%(notice the change of notation:
%$\gamma_i$ is now a component of $\Gamma$, not its strict transform by a sequence
%of blow-ups).  
If ${\mathcal F}$ is not $1$-dicritical, then up to a linear change of
coordinates we may assume that $x=0$ is not one of the lines of the
tangent cone of ${\mathcal F}$.  If, on the contrary, ${\mathcal F}$ is
$1$-dicritical, we may assume (after an analytic change of coordinates) that $x=0$ is not the tangent cone of any
$\gamma^j$ for $1 \leq j \leq q$, that $x=0$ is ${\mathcal F}$-invariant
and that the point defined by $x=0$ in $\pi_1^{-1}(0,0)$ is a regular point of $\mathcal{F}_1$.

For each of $\gamma^1,\dots, \gamma^q$,
%be the irreducible components of $\Gamma$ and
%denote
let $n_1,\dots, n_q$ be its corresponding multiplicity.
We denote $n = \mathrm{lcm} (n_1,\dots, n_q)$ and $\tau (x,y)= (x^{n},y)$.
\begin{rem}
\label{rem:ram1}
From the hypothesis on $\mathcal{F}$ and $\Gamma$ follows that
$\nu_0 (\tau^{*} {\mathcal F})= \nu_0 ({\mathcal F})$
and $\nu_0 (\tau^{-1}(\Gamma)) = \nu_0 (\Gamma)$.
\end{rem}
  Our choice of $\tau$ implies
that all the irreducible components of $\tau^{-1}(\Gamma)$ are smooth
and that there are exactly $\nu_0 (\Gamma)$ of them. 
  We are going to desingularize the curve $\tau^{-1} (\Gamma)$.  
In parallel to the notation of the previous section, 
% \ref{def:not},
we let $\pi_{\tau}=\pi^{\tau}_1 \circ \cdots \circ  \pi^{\tau}_r$ be the sequence
of blow-ups in the minimal desingularization of $\tau^{-1}(\Gamma)$,
we denote
$\tilde{\pi}^{\tau}_l=\pi_1^{\tau}\circ\cdots\circ\pi_l^{\tau}$,
$E_{\tau}$ the exceptional divisor of $\pi_{\tau}$, and $D_{\tau,l}$
the irreducible component of $E_{\tau}$ corresponding to
$\tilde{\pi}^{\tau}_l$, that is $D_{\tau,l}=(\pi_l^{\tau})^{-1}(P_{l-1})$.

% We are going to desingularize the curve
% $\tau^{-1}(\Gamma)$ by blowing up infinitely near points.  Such task
% amounts to making the irreducible components of $\tau^{-1}(\Gamma)$ pass
% through pairwise different points of the divisor of the desingularization
% process.

%   Let {$\pi_1,\dots, \pi_{k}$} be the blow-ups in the minimal
% desingularization of $\tau^{-1}(\Gamma)$.  Denote
% %$\tilde{\pi}_{l} = \pi_1 \circ\dots \circ \pi_l$,
% %$E_{l} = \tilde{\pi}_{l}^{-1}(0,0)$ and
% ${\mathcal G}_{l} = \tilde{\pi}_{l}^{*} ({\mathcal G}_0)$, with
% $\mathcal{G}_0 = \tau^{*}\mathcal{F}$.  We use the notations in Definition \ref{def:not}.
% %We also denote $\pi = \tilde{\pi}_k$, $E=E_k$ and ${\mathcal G}' = {\mathcal G}_k$.
  
\begin{defi}
  An irreducible component $D\subset E_{\tau}$ is {\it terminal} if $D$ is
  not the father of any other divisor:  assuming $D$ is the exceptional
  divisor of $\pi^{\tau}_l$ then no point of $D$ is a center for the
  blow-ups $\pi_{l+1}^{\tau},\dots, \pi_{k}^{\tau}$. Equivalently, $D$ has $-1$ self-intersection as a submanifold of $\pi_{\tau}^{-1} ({\mathbb C}^{2})$.  
  We denote by $\mathcal{I}_{\tau}$ the set of connected components of $\mathrm{Inv}(E_{\tau})$.
  
  We say that a component
  $H \in \mathcal{I}_{\tau}$ is terminal if it contains a terminal   irreducible component of $E_{\tau}$.  
\end{defi}  
  
\begin{rem}
\label{rem:ram2}
Since $\Gamma$ has multiplicity at least $2$, given any terminal component
$D$ of $E_{\tau}$, at least two irreducible components of
$\tau^{-1}(\Gamma)$ have strict transforms that intersect $D$.
\end{rem}  
 \begin{defi}
  \label{definition:reducible-multiplicity}
  % Given a curve $\Gamma$ whose irreducible
  % components are $\gamma_1,\dots, \gamma_q$.
  With the notations above, we define the {\it terminal virtual
    multiplicity} $\mu_{T}(\Gamma)$ of $\Gamma$ as the number of
  terminal irreducible components of   $E_{\tau}$.    We define the {\it
    divisorial virtual multiplicity} $\mu_{D} (\Gamma)$ as the number
  of irreducible components of   $E_{\tau}$   meeting the strict transform of
  $\tau^{-1}(\Gamma)$.
\end{defi}
\begin{rem} 
  The numbers $\mu_T (\Gamma)$ and $\mu_D (\Gamma)$ can be computed from
  the Puiseux expansions of $\gamma^1, \hdots, \gamma^q$.  Let the curve
  $\gamma^j$ have a Puiseux expansion $(t^{n_j}, c_j(t))$. Any power
  series of the form
\begin{equation*} 
\left( t, c_{j}(e^{\frac{2 \pi i l}{n_j}} t^{\frac{1}{n_j}}) \right)
\end{equation*} 
is also a parametrization of the same curve for $0 \leq l < n_j$.
The expression  
\begin{equation*} 
\left( t, c_{j}(e^{\frac{2 \pi i l}{n_j}} t^{\frac{n}{n_j}}) \right)
\end{equation*} 
provides all the parametrizations of the $n_j$ curves in  $\tau^{-1} (\gamma^j)$ by 
taking  $0 \leq l < n_j$ (recall that $n=\mathrm{lcm}(n_1,\ldots, n_q)$). Let ${\mathfrak C}_{j}$ be the set 
consisting of the power series expansions
$c_{j}(e^{\frac{2 \pi i l}{n_j}} t^{\frac{n}{n_j}})$ where $0 \leq l < n_j$ 
and ${\mathfrak C} = \cup_{j=1}^{q} {\mathfrak C}_j$.
The cardinal of ${\mathfrak C}$ is equal to $\nu_0 (\Gamma)$.
Consider the natural map $j_l: {\mathfrak C} \to J^{l}$ from ${\mathfrak C}$ to the set 
$J^{l}$ of $l$-jets of 
formal power series for $l \in {\mathbb N}$.
Given an $l$-jet $\sigma \in J^{l}$ we say that $\sigma$ is terminal 
if $j_{l}^{-1}(\sigma)$ contains at least two elements and 
$(j_{l+1})_{|j_{l}^{-1}(\sigma)}$ is injective. 
We say that  $\sigma$ is divisorial if $j_{l}^{-1}(\sigma)$ contains at least two elements and 
there exists $\sigma'$ in $j_{l+1} (j_l^{-1}(\sigma))$ such that 
$\sharp (j_{l+1})_{|j_{l}^{-1}(\sigma)}^{-1} (\sigma') = 1$.
The terminal (resp. divisorial) virtual multiplicity $\mu_{T} (\Gamma)$ 
(resp. $\mu_{D} (\Gamma)$) coincides with the number of terminal (resp. divisorial) jets. 

Given $1 \leq j \leq q$ there exists $l \in {\mathbb N} \cup \{0 \}$ such that 
all the fibers of $j_l :{\mathfrak C}_{j} \to J^{l}$  have $\nu_0 (\gamma^j)/\mu(\gamma^j)$
elements but $j_{l+1}: {\mathfrak C}_{j} \to J^{l+1}$ is injective.
As a consequence  $\mu_{T} (\Gamma)$, $\mu_{D} (\Gamma)$ 
and the definition of the virtual multiplicity $\mu (\Gamma)$  in terms of 
Puiseux characteristic exponents coincide for an irreducible curve $\Gamma$.
Moreover, we obtain
\begin{equation*} 
\mu_{D} (\Gamma) \geq \mu_{T} (\Gamma) \geq \max (\mu (\gamma^1), \hdots, \mu (\gamma^q)). 
\end{equation*}
We can interpret $\mu_{T} (\Gamma)$ and $\mu_{D}(\Gamma)$ as generalizations of 
the virtual multiplicity to the reducible case. 
\end{rem}
\begin{example}
  \label{ex:terminal-and-divisorial}
  Consider the union $\Gamma=\gamma_1\cup\gamma_2\cup\gamma_3\cup\gamma_{4}$, where each $\gamma_i$ corresponds to the arrows in the dual graph given in Figure \ref{fig:terminal-and-divisorial-1}. For instance, $\gamma_1=(t^3,t^4)$, $\gamma_2=(t^6,t^{8}+t^{10}+t^{11})$, $\gamma_3=(t^6,t^8+t^{10}+t^{11}+t^{13})$, and $\gamma_4=(t^6,t^8+t^{10}+t^{11}-t^{13})$.
\begin{figure}[h!]
 \centering
 \includegraphics{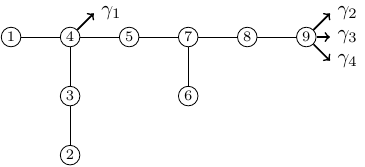}
\caption{The curve $\Gamma$ of Example \ref{ex:terminal-and-divisorial} is the union of the branches given by the $\gamma_i$ in the diagram.}
\label{fig:terminal-and-divisorial-1}
\end{figure}

The map $\tau$ is $\tau(u,y)=(u^6,y)$, and the dual graph of the desingularization of $\tau^{-1}(\Gamma)$ is schematically shown in Figure \ref{fig:terminal-and-divisorial-2} (all the centers $P_j$ are trace points). There are   $6$   terminal irreducible components of $E_{\tau}$, and   $9$   components meeting the strict transform of $\tau^{-1}(\Gamma)$, so that   $\mu_T(\Gamma)=6$ and $\mu_D(\Gamma)=9$.   
 \begin{figure}[h!]
  \centering             
  \includegraphics{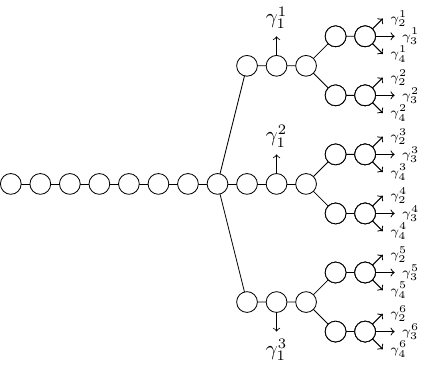}
  \caption{Dual graph of the desingularization of $\tau^{-1}(\Gamma)$, for $\Gamma$ as in Figure \ref{fig:terminal-and-divisorial-1}. No blow-up center is a corner.}\label{fig:terminal-and-divisorial-2}
\end{figure}
\end{example}
% \begin{defi}
%    Let $D$ be either an irreducible component of $E$ or an element
%   of $\mathcal{I}$. The number of irreducible components of
%   $\tau^{-1}(\Gamma)$ whose strict transform meets $D$ will be denoted
%   by $c_D$.
% \end{defi}
The next result provides lower bounds for $\nu_{0}({\mathcal F})$ in terms of $\Gamma$ and its desingularization. It implies Theorem \ref{teo:genr}. 
%We denote by $\mathcal{I}_{\tau}$ the set of connected components of $\mathrm{Inv}(E_{\tau})$.
\begin{pro}
\label{pro:genr}
Let ${\mathcal F}$ be a germ of holomorphic foliation defined in a
neighborhood of $0$ in ${\mathbb C}^{2}$. Let $\Gamma$ be a germ of
singular invariant curve. Given $H\in \mathcal{I}_{\tau}$, let $c_H$ be the number of irreducible components of $\tau^{-1}(\Gamma)$ whose strict transform meets $H$. We have
  \[
    \nu_0 ({\mathcal F})  \geq N  + \sum_{H \in {\mathcal I}}  (\max(c_H,1) -1)
    \geq
    {\max \left(\mu_{T} (\Gamma), \frac{\mu_{D}(\Gamma)}{2} \right)} 
  \]
  where $N$ is the number of dicritical irreducible components of $E_{\tau}$. 
\end{pro}
Notice how the first inequality depends on $\mathcal{F}$, whereas the second is completely independent of it, and requires only geometric information on $\Gamma$.

\begin{proof}
  Let $\mathcal{G}_0=\tau^{\ast}\mathcal{F}$, which satisfies
  $\nu_0(\mathcal{G}_0)=\nu_0(\mathcal{F})$, and
  $\mathcal{G}_l=(\pi_l^{\tau})^{\ast}(\mathcal{G}_0)$ for
  $1\leq l \leq r$. Notice that $w(D)=1$ for any irreducible component
  $D\subset E_{\tau}$, as $\pi^{\tau}_l$ is never the blow-up of a
  corner point.  Hertling's formula \cite{Hertling:2000} in this case
  becomes
\begin{equation*}
  \nu_0 ({\mathcal G}_0)  =
  \sum_{D_{\tau, j} \subset E_{\tau}}
  \sum_{P \in D_{\tau, j}} \kappa_P ({\mathcal G}_k, D_{\tau, j})  + 
  \sum_{D_{\tau, j} \not\subset\mathrm{Inv}(E_{\tau})} (2 - v_{\overline{d}} (D_{\tau, j})) -1 
\end{equation*}
Given $H \in \mathcal{I}_{\tau}$, the following
two inequalities hold:
\begin{equation*}
   \sum_{\substack{D \in H\\P \in D}} \kappa_{P} ({\mathcal G}_k,D) \geq c_H \ \  \mathrm{and} \ \
  \sum_{\substack{D \in H\\P \in D}} \kappa_{P} ({\mathcal G}_k,D) \geq 1 .
\end{equation*} 
The first one   is satisfied   because every strict transform of an irreducible
component of $\tau^{-1}(\Gamma)$ intersects $H$ at a singular point of
$\mathcal{G}_k$. The second one is a consequence of \cite[Proposition
3.7]{Cano-Fortuny-Ribon:2020} (see also Remark \ref{rem:hidden}).  We
deduce
\begin{equation*}
 \nu_0 ({\mathcal F}) = \nu_0 ({\mathcal G}_0) \geq
 \sum_{D_{\tau, j} \not \subset \mathrm{Inv}(E_{\tau})} (2 - v_{\overline{d}} (D_{\tau,  j})) -1
 + \sharp {\mathcal I}_{\tau}
 + \sum_{H \in {\mathcal I}_{\tau}}  (\max(c_H,1) -1).
\end{equation*}
Let $D_{\tau, j} \not\subset\mathrm{Inv}(E_{\tau})$ be a dicritical irreducible component of $E_{\tau}$, and let $H_{j,1}, \ldots, H_{j,p}$ be the elements of $\mathcal{I}_{\tau}$ intersecting $D_{\tau, j}$ which appear after blowing-up a center in $D_{\tau, j}$ (technically, if $D_{\tau, l} \subset H_{j,r}$ then $l>j$). By definition, $v_{\overline{d}}(D_{\tau,j})$ is at most $p+1$, except for $j=1$, where $v_{\overline{d}}(D_{\tau, 1})\leq p$. As a consequence, if $N$  is the number of non-invariant irreducible components of $E_{\tau}$, we have:
\begin{equation*}
  \sum_{D_{\tau, j} \not\subset\mathrm{Inv}(E_{\tau})}(2 -v_{\overline{d}}(D_{\tau, j}))
  -1 +
  \sharp \mathcal{I}_{\tau} \geq N
\end{equation*}
% Set $\hat{v}_{\overline{d}} (D_1)=1$.  Given an irreducible component
% $D_j$ of $E_{\tau}$ with $j>1$ there exists a unique $D_l$ with $l < j$ such that
% $D_l \cap D_j \neq \emptyset$.  We define $\hat{v}_{\overline{d}} (D_j)$
% as $1$ if $D_l$ is invariant and $0$ otherwise.  Every element $H$ of
% ${\mathcal I}_{\tau}$ was created by blowing-up either the origin or a
% non-corner point $P$ in a non-invariant component of $D_j$ of
% $E_{\tau,l}=(\pi^{\tau}_{l})^{-1}(0,0)$.  By adding $1$ to
% $2 - v_{\overline{d}} (D_j)$ for every $H \in {\mathcal I}_{\tau}$ whose
% ``father" is the same $D_j$ we obtain
% \begin{equation*}
%   \nu_0 ({\mathcal F}) = \nu_0 ({\mathcal G}_0) \geq
%   \sum_{D_j \not \subset \mathrm{Inv}(E_{\tau})}
%   (2 - \hat{v}_{\overline{d}} (D_j)) 
%   + \sum_{H \in {\mathcal I}_{\tau}}  (\max(c_H,1) -1).
% \end{equation*} 
% Since $\hat{v}_{\overline{d}} (D_j) \leq 1$ for any non-invariant irreducible component $D_j$ of $E_{\tau}$,
so that:
\begin{equation}
\label{equ:npc}
 \nu_0 ({\mathcal F}) = \nu_0 ({\mathcal G}_0) \geq
 N + \sum_{H \in {\mathcal I}_{\tau}}  (\max(c_H,1) -1)
\end{equation}
which is the first inequality of the statement.

In order to prove the second inequality, given $H \in {\mathcal I}_{\tau}$, let $t_{H}$ be the number of terminal
components in $H$. If $t_H \geq 1$ then $c_H \geq 2t_H$ because every
terminal component meets at least two branches of the strict transform
of $\tau^{-1}(\Gamma)$ (Remark \ref{rem:ram2}); hence
$c_H -1 \geq 2t_H -1 \geq t_H$. This
implies that
$$
N + \sum_{H \in {\mathcal I}_{\tau}} (\max(c_H,1) -1)
\geq \mu_{T}(\Gamma).
$$

We now compare $\nu_0(\mathcal{F})$ and ${\mu}_{D}(\Gamma)/2$. Denote by
${\mathcal T}_{\tau}$ the set of terminal elements of ${\mathcal I}_{\tau}$.
Given $H \in {\mathcal I}_{\tau}$, we denote by $d_H$ the number of irreducible
components of $H$ meeting the strict transform of $\tau^{-1}(\Gamma)$. The
inequality $c_H \geq d_H$ always holds. Moreover $c_H - 1 \geq d_H$ is
satisfied for any $H \in {\mathcal T}_{\tau}$ 
%since any terminal divisor $D$ contained in $H$ satisfies $c_D \geq 2$ 
  by Remark \ref{rem:ram2}.    We
have
    \[
    2 \nu_0 ({\mathcal F})  \geq
    2N + 2 \sum_{H \in {\mathcal I}_{\tau}}  (\max(c_H,1) -1) = \star
  \]
  which, expanding all the terms and using the relations between $c_H$ and $d_H$ above, gives
  \[
    \star
    \geq N + 2 \sum_{H \in {\mathcal T}_{\tau}} d_H +
    \sum_{\substack{H
        \in {\mathcal I}_{\tau} \setminus {\mathcal T}_{\tau} \\
        c_H \geq 1}} d_H + (N - M) + \sum_{\substack{H \in {\mathcal
          I}_{\tau} \setminus {\mathcal T}_{\tau} \\ c_H \geq 1}} d_H - M
 \]
 where
 $M=\sharp \{ H \in {\mathcal I}_{\tau} \setminus {\mathcal T}_{\tau} : c_
 H \geq 1 \}$.  Any non-terminal $H \in {\mathcal I}_{\tau}$ has a non-invariant
   irreducible component of $E_{\tau}$  
 as one of its adjacent sucessors, i.e. there exists a
 non-invariant irreducible component of $E_{\tau}$ that was the result of a
 blow-up with center in a point in $H$.  Therefore $M \leq N$ holds and we
 obtain
 \[
   2 \nu_0 ({\mathcal F})   \geq N + 2 \sum_{H \in {\mathcal T}_{\tau}} d_H +
   \sum_{\substack{H \in {\mathcal I}_{\tau} \setminus {\mathcal T}_{\tau} \\c_H \geq 1}} d_H  \geq 
   {\mu}_{D} (\Gamma) + \sum_{H \in {\mathcal T}_{\tau}}
   d_H   \geq {\mu}_{D} (\Gamma)
\]
 as desired. Notice that if $\mathcal{I}_{\tau}=\emptyset$ the result is straightforward.
 \end{proof}
%\begin{cor}
%\label{cor:genr}
%If $\Gamma$ is a singular curve invariant for a germ of
%holomorphic foliation ${\mathcal F}$ defined in a neighborhood of the origin
%of ${\mathbb C}^{2}$, then
%$\nu_{0} ({\mathcal F}) \geq \mu (\Gamma)$.
%\end{cor}
\begin{rem}
 Theorem \ref{teo:genr} is the analogue in the reducible case
 of \cite[Theorem 3.1]{Cano-Fortuny-Ribon:2020} (where $\mu_T(\Gamma)=\mu_D(\Gamma)=\mu(\Gamma)$).
\end{rem}
\begin{rem}
The inequality 
$\nu_{0} ({\mathcal F}) \geq \mu (\Gamma)$ is sharp if 
$\Gamma$ is   a germ of irreducible 
invariant  curve   \cite[Remark 3.11]{Cano-Fortuny-Ribon:2020}.
By considering a ramification we obtain examples of reducible invariant curves $\Gamma$
whose branches are smooth and such that $\nu_{0} ({\mathcal F}) = \mu_{T} (\Gamma)$.
\end{rem}
\begin{rem}
Consider a reducible curve $\Gamma$ consisting of $n$ smooth curves
$\gamma^1, \hdots, \gamma^n$ with $n \geq 2$. 
Assume that the   exceptional divisor $E$ of the  
 desingularization of $\Gamma$ has $n-1$   irreducible components   $D_1, \hdots, D_{n-1}$
and only the last one is terminal.
This is for instance the situation if $\gamma^j = \{ y = x^{j} \}$ for $1 \leq j \leq n$. 
Suppose that 
\begin{itemize}
\item $D_j$ is invariant if and only if $j$ is odd;
\item Let $P \in E$. Then 
$\kappa_{P} ({\mathcal F}_{n-1},\Gamma_{n-1})=1$ if $(E,P)$ is invariant and 
$P \in \Gamma_{n-1}$ and 
$\kappa_{P} ({\mathcal F}_{n-1},\Gamma_{n-1})=0$ otherwise.
\end{itemize}
Such an example   of foliation ${\mathcal F}$    can be built by using the realization theorem of Lins Neto 
\cite{Lins-Neto:1987}.
The self-intersection of $D_j$ is $-2$ if $j <n-1$ and $-1$ if $j=n-1$. 
The foliation is regular and
transverse to every  even divisor   $D_j$    in any point. The odd divisors $D_j$ with $j < n-1$ have 
  a unique (nondegenerate irreducible) singular point.
The divisor $D_{n-1}$ (if $n-1$ is odd) has  two
(nondegenerate irreducible) singular points. Such configurations are easy to build.    Once ${\mathcal F}$ is fixed, we  
choose the unique curve through every singular point in $\mathrm{Inv} (E)$ and 
for any non-invariant   irreducible component  $D_j$ of $E$   we choose an invariant curve through a trace point of $D_j$
(or two trace points if $j=n-1$).  In this way we obtain a foliation ${\mathcal F}$ that leaves
invariant a curve $\Gamma$ with the properties described above. In this case, it is easy to 
see that    $\mu_{T}(\Gamma)=1$, $\mu_{D}(\Gamma)=n-1$ and  
\[ \nu_0 ({\mathcal F}) = N + \sum_{H \in {\mathcal I}}  (\max(c_H,1) -1) = 
 \left\lceil \frac{n-1}{2} \right\rceil = 
 \left\lceil \frac{\mu_{D}(\Gamma)}{2} \right\rceil,  \]
cf. Equation (\ref{equ:npc}), where $\lceil s \rceil$ is the smallest integer greater or 
equal than $s$.  In particular Theorem \ref{teo:genr} is both interesting (since it provides 
non-trivial lower bounds for $\nu_0 ({\mathcal F})$) and sharp.  
\end{rem}

\section{Isolated invariant curves} 
In the previous section we bounded from below the multiplicity of a
singular foliation $\mathcal{F}$ in terms of invariants of a singular
curve $\Gamma$ consisting of leaves of $\mathcal{F}$. These invariants
depend only on the desingularization of $\Gamma$; in order to obtain a
lower bound in terms of the multiplicity of $\Gamma$, we need to
require $\Gamma$ to be composed of invariant branches that are
isolated somehow: one can get the multiplicity of $\Gamma$ arbitrary
large choosing branches which meet a dicritical component
of the exceptional divisor transversely. The first definition in that
direction is ``not belonging to a dicritical family'' (see
\cite{Corral-Fernandez:2006}).
%, specifically, \emph{isolated branches}.
% \textcolor{red}{Vamos a dar una cota inferior de un germen de foliacion en funcion de la 
% multiplicidad de una curva integral aislada. Infelizmente la propiedad de ser aislada no es
% invariante por ramificacion, así que no se puede deducir en principio de la seccion anterior.
% Algunas cosas de la seccion anterior estan repetidas. 
% He puesto esto aqui por falta de un lugar mejor.}

% We now tackle the case of a reducible curve $\Gamma$ invariant
% by a germ of holomorphic foliation ${\mathcal F}$
% defined in a neighborhood of the origin of ${\mathbb C}^{2}$.
% %\pedro{}Denote\pedro{} by
% %$\gamma_1, \gamma_2,\dots, \gamma_q$ the irreducible components of $\Gamma$.
% As in the previous sections, we want to obtain a
% lower bound for the multiplicity of $\Gamma$
% in terms of data associated to it.
\begin{defi}\label{def:isolated-corral}
  Let $\gamma$ be an irreducible ${\mathcal F}$-invariant curve. It is
  called \emph{isolated} if there is no birational morphism
  ${\pi} :(M,D) \to ({\mathbb C}^{2},0)$ such that the strict transform of
  $\gamma$ intersects $D$ transversally at a non-corner regular point of
  the lifted foliation $\pi^{\star}\mathcal{F}$. We say that a reduced
  invariant curve $\Gamma$ is isolated if all its irreducible components
  are isolated.
\end{defi}
However, this notion is too restrictive. We shall see that one only
needs to rule out specific dicritical families related to $\Gamma$
and, what is more: some non-isolated irreducible components may be
acceptable.

Recall that a \emph{normal-crossings divisor at a point $P$} is the
union of zero, one or two non-singular irreducible curves containing
$P$, and in the last case, they are mutually transverse. Let
$\hat{\Gamma}$ be a singular curve invariant for $\mathcal{F}$, and
let $\pi :M \to ({\mathbb C}^{2},0)$ be its minimal desingularization
with $E=\pi^{-1}(0)$. Given an irreducible component $\gamma$ of
$\hat{\Gamma}$, let $P_{\gamma}$ the point at which the strict
transform of $\gamma$ meets $E$, and $D_{\gamma}$ the irreducible
component of $E$ to which $P_{\gamma}$ belongs.
\begin{defi}
\label{def:weak_isolation}
 We say that a   singular curve  
 $\hat{\Gamma}$ is {\it weakly isolated} for ${\mathcal F}$ if there is a normal-crossings divisor $\overline{\Gamma}$ such that $\hat{\Gamma}=\Gamma \cup \overline{\Gamma}$, and $\Gamma\cap \overline{\Gamma}\subset \left\{ 0 \right\}$, with:
   \begin{itemize}
  \item  $\kappa_{P_\gamma} (\pi^{*} {\mathcal F}, D_{\gamma}) \geq 1$  for any irreducible 
  component $\gamma$ of $\Gamma$.
  \end{itemize}
  Any irreducible component $\eta$ of $\hat{\Gamma}$ with  
  $\kappa_{P_\eta} (\pi^{*} {\mathcal F}, D_{\eta}) = 0$ will be called \emph{null}, the other components of $\hat{\Gamma}$ will be called \emph{non-null}. Null components are, by definition, included in $\overline{\Gamma}$.
\end{defi}
Thus, a weakly isolated curve is composed of an ``important'' part
$\Gamma$ and a ``discardable'' one, $\overline{\Gamma}$. Later on the
roles of these two parts will become clear. See Figure
\ref{fig:weakly-isolated-blow-up} for an example.
\begin{rem}
A singular isolated integral curve is weakly isolated (cf. Remark \ref{rem:zero_index}).
\end{rem}
Weak isolation is  a powerful concept since it is invariant by blow-ups.
\begin{pro}
\label{pro:weak_invariance}
Let $\hat{\Gamma}$ be weakly isolated for ${\mathcal F}$ and $\pi_1$ the
blow-up of the origin. Let $P \in \pi_{1}^{-1} (0)$ and
$\hat{\Gamma}_{P}$ be the union of the invariant irreducible components of
$\pi^{-1} (\Gamma)$ (the \emph{total} transform) containing $P$. Then
$\hat{\Gamma}_{P}$ is a weakly isolated curve for
$\pi_{1}^{*} {\mathcal F}$ if $\hat{\Gamma}_P$ is singular.
\end{pro}
\begin{proof} 
  Write $\hat{\Gamma}=\Gamma\cup \overline{\Gamma}$ where
  $\overline{\Gamma}$ is the normal-crossings divisor formed by the null
  components of $\hat{\Gamma}$, and $\hat{\pi}=\pi_1\circ \pi^{\prime}$ a
  minimal desingularization of $\hat{\Gamma}$. Notice that
  $\overline{\Gamma}$ may perfectly be empty. Denote by
  $\overline{\Gamma}_P$ the strict transform of $\overline{\Gamma}$ at
  $P$, and let $a$ be the number of irreducible components of
  $\overline{\Gamma}_{P}$. As $\overline{\Gamma}$ has normal crossings,
  $a\leq 1$, and if $a=1$ then $\overline{\Gamma}_P$ is transverse to
  $D_1=\pi_1^{-1}(0,0)$.  Let $\pi_P$ be the minimal desingularization of
  $\hat{\Gamma}_P$. By definition, $\pi_P$ (as a sequence of point
  blow-ups) is contained in the sequence of point blow-ups corresponding
  to $\pi^{\prime}$.  We consider two cases, depending on the invariance
  of $D_1$ by $\mathcal{F}_1$.
 
  \emph{Case $D_{1}$ invariant.}  The desingularization provided by
  $\pi'$ of $\hat{\Gamma}_{P}$ and $\pi_P$ coincide because
  $D_1\subset \hat{\Gamma}_P$.  Thus, the null irreducible components
  of $\hat{\Gamma}_{P}$ are $\overline{\Gamma}_P$ (if $a \geq 1$) and
  maybe $D_1$ by hypothesis.  Therefore, $\hat{\Gamma}_P$ is weakly
  isolated.
 
  \emph{Case $D_{1}$ non-invariant.}  The only blow-ups in
  $\pi^{\prime}$ not belonging to $\pi_P$ must be of points in the
  successive strict transforms of $D_1$. Thus, if non-empty, this
  sequence corresponds to a single smooth irreducible component
  $\gamma^{\prime}$ of $\hat{\Gamma}_P$ tangent to $D_1$, and this
  curve might be null for $\pi_1^{\ast}\mathcal{F}$ (as it is already
  smooth at $P$).

Anyway,
   $\hat{\Gamma}_P$ is indeed weakly isolated for $\pi_1^{\ast}\mathcal{F}$, as the only possible null components are $\gamma^{\prime}$ and $\overline{\Gamma}_P$, which are transverse.
\end{proof}
\begin{example}\label{ex:weakly-isolated}
  Consider a foliation having the curves   $\overline{\Gamma} \equiv (y=0)$ and
  $\Gamma \equiv (y^2-x^3=0)$   as separatrices, and such that if $\pi$
  is the desingularization of   $\hat{\Gamma} := \Gamma \cup \overline{\Gamma}$,   then
  $D_1$ and $D_2$ are non-invariant,   $\overline{\Gamma} \cap D_2$   is non-singular for
  $\mathcal{F}_2$, and $D_3$ is invariant (see Figure
  \ref{fig:weakly-isolated-blow-up}).   The curve
  $\hat{\Gamma}$ is weakly isolated but not isolated. Moreover,   the strict transform of $\hat{\Gamma}$ by $\pi_1$  
  is also weakly isolated.
\end{example}
\begin{figure}[h!]
  \centering
   \includegraphics{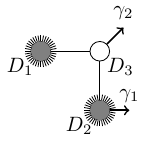}
\caption{See Example \ref{ex:weakly-isolated}. If $\mathcal{F}_{3}$ is
  reduced and $\gamma_1, \gamma_2$ are invariant, then
  $\Gamma=\gamma_1\cup \gamma_2$ is weakly isolated but not
  isolated. Its strict transform $\pi_1^{*}(\Gamma)$ is weakly
  isolated but not isolated.}\label{fig:weakly-isolated-blow-up}
\end{figure}

In this section, $\hat{\Gamma}=\gamma^1 \cup \ldots \cup \gamma^q$
will be a weakly isolated curve for $\mathcal{F}$, and
$\pi=\pi_1\circ\cdots\circ\pi_k$ will denote its minimal
desingularization, the rest of the notation being as in Section
\ref{sec:setting}. Recall that $\hat{\Gamma}_k$ meets
$E_k$ transversely. We remark again that we only need to study the
desingularization of $\hat{\Gamma}$, and not of $\mathcal{F}$. The
only hypothesis on $\mathcal{F}$ is that $\hat{\Gamma}$ is weakly isolated.
%\begin{defi}
% We consider the notations in Definition \ref{def:not}; in particular
% $\pi_1,\dots, \pi_{k}$ will denote the blow-ups in the
% desingularization of $\Gamma$,
% $\tilde{\pi}_{l} = \pi_1 \circ\dots \circ \pi_l$ and $P_0 =(0,0)$. 
% As above, $\pi_1,\dots, \pi_{k}$ will denote the blow-ups in the
%desingularization of $\Gamma$,
%$\tilde{\pi}_{l} = \pi_1 \circ\dots \circ \pi_l$, $P_0 =(0,0)$ 
%and the remaining notations in Definition \ref{def:not}.
% We say that $\Gamma_k$ is desingularized if 
% $\gamma_{k}^{j}$ is smooth and interscts transversally $E$ at a trace
% point for $1 \leq j \leq q$ and the strict transforms
% $\gamma_{k}^{1}, \hdots, \gamma_{k}^{q}$ are pairwise disjoint.

% As elsewhere, we are only concerned about the desingularization of $\Gamma$ and its relation to $\mathcal{F}$, not the desingularization of $\mathcal{F}$. This is one of the main features of our approach.

We do not follow the ramification approach of Section \ref{sec:last},
as the isolation property for a curve $\Gamma$ is not invariant under
ramification even if the ramification locus is different from
$\Gamma$.  Moreover, our approach to the study of the global
Poincar\'{e} problem through local methods requires calculating
vanishing orders of foliations along invariant branches $\gamma$ of an
invariant curve $\Gamma$. Later on, we shall use an iterative formula
(Eq. \ref{equ:ziter}) to calculate these vanishing orders. Since the
formula collects data associated to a desingularization of $\Gamma$,
it is convenient to avoid ramifications.

%the center of
%$\pi_{l+1}$, $D_{l}= {\pi}_{l}^{-1}(P_{l-1})$, $E_0 = \{(0,0)\}$, 
%$E_{l} = \tilde{\pi}_{l}^{-1}(0,0)$,
%${\mathcal F}_{l} = \tilde{\pi}_{l}^{*} ({\mathcal F}_{l-1})$, 
%${\mathcal F}' ={\mathcal F}_k$ and $\pi = \tilde{\pi}_k$ where
%${\mathcal F}_0 = {\mathcal F}$ and $P_0 =(0,0)$.  
%We denote by $\mathrm{Inv} (E)$ the union of the invariant irreducible components of $E$.
%\begin{rem}
%\label{rem:hidden}
%Consider an isolated invariant branch $\gamma$ of $\Gamma$. 
%We have $\kappa_{P} ({\mathcal F}_k, \gamma_k) \geq 1$.
%\end{rem}   
 
 Our approach to providing lower bounds for the multiplicity of a foliation ${\mathcal F}$ 
 in terms of the multiplicity of a weakly isolated curve consists in   
 dividing $E_k$ into connected unions of
irreducible components: each starting with the exceptional divisor
corresponding to the blow-up of a trace point of a non-invariant
component. In order to carry out this division, we require some
nomenclature.

\begin{defi}\label{def:descendant}
  Let $0\leq l, l^{\prime} \leq k-1$ be indices of two centers of
  $\pi$. We say that $P_{l^{\prime}}$ is a
  \emph{descendant} of $P_l$ (and $P_l$ is an \emph{ancestor} of
  $P_{l\prime}$) if
  $\pi_{l+1}\circ \cdots \circ \pi_{l^{\prime}}(P_{l\prime})=P_l$
  or $l = l^{\prime}$.  
\end{defi}
For simplicity in later arguments, we consider any point $P_l$ both an
ancestor and a descendant of itself.

In Hertling's formula \eqref{equ:Hertling}
 \begin{equation*}
\nu_{P_0} ({\mathcal F})  +1 = \sum_{D_j\subset E_k} \sum_{P \in D_j} w(D_j) \kappa_P ({\mathcal F}_k, D_j)  
 +  \sum_{D_j  \not \subset \mathrm{Inv}(E_{k})} w(D_j)  (2 - v_{\overline{d}} (D_j)) 
\end{equation*}
the \emph{problematic} terms are the ones in the last summation, which
correspond to non-invariant components of the exceptional
divisor. These points can be of two different kinds: either they arise
for the first time from trace points, or they do from corners. The
former are the key ones   to   divide $E_k$ (and, as a
consequence, Hertling's formula) into \emph{controllable}
parts. Properly speaking: given $l\in \left\{ 0,\ldots, k \right\}$,
we shall denote by $D(P_l)$ the set of irreducible components of $E_l$
containing $P_l$. We say that $P_l$ is a \emph{separating center}
(\emph{s.c.}  for brevity) if either $l=0$ (whence
$D(P_0)=\emptyset)$) or if $D(P_l)$ is a singleton and its unique
element is non-invariant for $\mathcal{F}_l$. Using this notion, we
can divide the exceptional divisor $E_k$ into connected sets each
starting ``immediately after'' a separating center, as in Figure
\ref{fig:proof-reducible-case}, using the following definition:
\begin{defi}\label{def:controllable-components}
  Let $0\leq l \leq k-1$ be an index such that $P_l$ is a separating
  center. We shall denote by $\mathcal{D}_l$ the set of divisors
  $D_{l^{\prime}+1}$, with $l^{\prime} \geq l$, such that $P_{l^{\prime}}$ is a descendant of $P_l$
  and $P_l$ is the unique separating center among the $P_{\ell}$ that are both descendants
  of $P_l$ and ancestors of $P_{l^{\prime}}$.   This gives a partition
  of $\left\{ 0,\ldots, k-1 \right\}$ (or, equivalently, a subdivision
  of $E_k$ into connected unions of irreducible components):
  \begin{equation*}
    E_k = \bigcup_{P_l\ \mathrm{s.c.}} \mathcal{D}_l.
  \end{equation*}
  See Figure \ref{fig:proof-reducible-case} for an example.
\end{defi}
The point $P_0$ is always a separating center by convention, so that
$\mathcal{D}_0$ always contains $D_1$ at least. Each set
$\mathcal{D}_l$ is, essentially, a \emph{controllable} part of $E_{k}$
in Hertling's formula, as the next lemmas show.

From now on, $\hat{\Gamma}$ will denote a weakly isolated curve for
$\mathcal{F}$ with $\hat{\Gamma}=\Gamma\cup \overline{\Gamma}$, where
$\overline{\Gamma}$ is the union of the null components of
$\hat{\Gamma}$. Given a separating center $P_l$, the curves
$\gamma^{l,1},\ldots, \gamma^{l,n_l}$    are    the irreducible
components of $\Gamma$   whose strict transform $\gamma_k^{l,j}$ satisfies  
$\gamma_k^{l,j}\cap \mathcal{D}_l\neq \emptyset$, and
$\overline{\gamma}^{l,1},\ldots, \overline{\gamma}^{l,m_l}$   are   the
irreducible components of $\overline{\Gamma}$ with
$\overline{\gamma}_{k}^{l,j}\cap\mathcal{D}_l\neq\emptyset $.

\begin{figure}[h!]
  \centering
  \scalebox{0.8}{
  \includegraphics{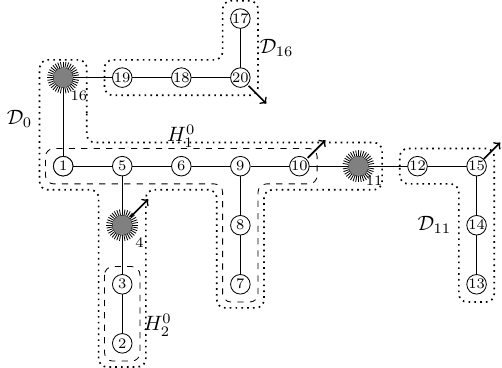}}
  \caption{Separating centers and the components $\mathcal{D}_l$.  Each
  component of the exceptional divisor $E_{k}$ is numbered according
  to its order of apparition. Arrows indicate irreducible components
  of $\hat{\Gamma}$.  Notice how
  ${\mathcal D}_0\supset H_1^0 \cup H_2^0$ contains two connected
  components of $\mathcal{I}$ and three non-invariant irreducible
  components of $E_{k}$.
  %Finally, the branch passing through
  %$D_4$ meets $D_4$ at a singular point of $\mathcal{F}_k$ necessarily
  %(as $\hat{\Gamma}$ is weakly isolated).}
  }
\label{fig:proof-reducible-case}
\end{figure}
If $P_l$ is a separating center, the component $\mathcal{D}_l$   satisfies   $\mathcal{D}_l\cap D(P_l)\neq \emptyset$ and
$\pi_k\circ\cdots\circ\pi_{l+1}(\mathcal{D}_l)=P_l$ and such properties characterize $\mathcal{D}_l$.   In Hertling's
formula \eqref{equ:Hertling}, the sum corresponding to non-invariant
components:
\begin{equation}\label{eq:dicritical-part}
  \sum_{D_j  \not \subset \mathrm{Inv}(E_{k})}
  w(D_j) (2 -v_{\overline{d}}(D_j))
\end{equation}
is now best divided according to the following conditions:
\begin{enumerate}
\item For each separating center $l$ with $l>0$, there is a single
  irreducible component $\overline{D}(P_l)\in\mathcal{D}_l$ with
  $\overline{D}(P_l)\cap D(P_l)\neq \emptyset$.
\item By convention, we set $\overline{D}(P_0)=\emptyset$ and
  $w(D(P_0))=1$.
\item Finally, all the other intersections between a non-invariant
  component $D_j$ and another component of $E$ arise, by
  definition,    with components of the form $\mathcal{D}_r$ where $P_r$ is an ancestor of $P_{j-1}$.
  There are at most two such components.  
\end{enumerate}
Using these three properties, recalling the definition of
\emph{non-dicritical} valence $v_{\overline{d}}(D_j)$ (Definition
\ref{def:non-dicritical-valence}), for each separating center $P_l$:
we set $\delta(P_l)=1$ if either $l=0$ or $\overline{D}(P_l)$ is
invariant, and $\delta(P_l)=0$ otherwise. Define, for any
non-invariant $D_j$ (for the sake of simplicity, and where an empty
summation is $0$):
\begin{equation*}
  v_{\overline{d}}^a(D_j) = v_{\overline{d}}(D_j) -
  \sum_{\substack{P\in D_j\\P\mathrm{\ s.c.}}}\delta(P).
\end{equation*}
\begin{rem} \label{rem:ancestor_valence}
The value $v_{\overline{d}}^a(D_j)$ is the part of the non-dicritical valence of $D_j$ related to its ancestors, i.e. 
to components ${\mathcal D}_r$ where $P_r$ is an ancestor of $P_{j-1}$.
It always satisfies $ v_{\overline{d}}^a(D_j) \leq c \leq 2$ where $c$ is the number of irreducible components of $E_{j-1}$ containing the point
$P_{j-1}$. For $j=1$, one always has $c=0$.
\end{rem}  
With this convention, the sum \eqref{eq:dicritical-part} can be
regrouped in the following way:
\begin{equation}
  \label{eq:dicritical-part-divided}
  \begin{split}
  &\sum_{D_j \not \subset \mathrm{Inv}(E_k)} w(D_j) (2
  -v_{\overline{d}}(D_j))
  =\\
  &1 + \sum_{P_l \mathrm{\ s.c.}}\bigg(
    - \delta(P_l) w(D(P_{l}))\bigg) +
  \sum_
  {D_j \not \subset \mathrm{Inv}(E_k)}
  w(D_j) (2 - v_{\overline{d}}^{a} (D_j)).
\end{split}
\end{equation}
Notice the term $1$ in the right hand side of Equation
(\ref{eq:dicritical-part-divided}): it is required in order to remove
the ``spurious'' $-\delta(P_0)w(D(P_0))$, as
$D(P_0)=\emptyset$. However, it is exactly what will cancel the $+1$
in Hertling's formula \eqref{equ:Hertling}.

At this point, we proceed to study each component $\mathcal{D}_l$
in isolation.
The positive contributions in Hertling's formula \eqref{equ:Hertling}
pertaining to a specific component $\mathcal{D}_l$ are easily bounded
in terms of either $\hat{\Gamma}$ or $P_l$:

\begin{lem}
  With the notations above, if $\hat{\Gamma}$ is weakly isolated for
  $\mathcal{F}$ and $P_l$ is a separating center, then
\begin{equation}
    \label{eq:non-inv-branches}
    \sum_{\substack{D\in {\mathcal D}_l\\ P\in D}} w(D) \kappa_P(\mathcal{F}_k,D)
    \geq \sum_{j=1}^{n_l} \nu_0(\gamma^{l, j}),
  \end{equation}
  Also, for any $H\in \mathcal{I}$ with $H\subset \mathcal{D}_l$:
 \begin{equation}
 \label{equ:contribh}
    \sum_{D_j \in H} \sum_{P \in D_j} w(D_j) \kappa_P ({\mathcal F}_k, D_j) \geq w(D(P_l)).
\end{equation}
\end{lem}
\begin{proof}
  The first inequality holds because
  $\kappa_{P_j^l} (\mathcal{F}' ,D)\geq 1$ at
  $\{P_j^l\}=\gamma_{k}^{l,j} \cap E$ by hypothesis and
  $\nu_0(\gamma^{l,j}) = w (D(P_j^l))$, where $D(P_j^l)$ is the unique
  irreducible component of $E$ that contains $P_{j}^{l}$.
  The second one follows from   Remark \ref{rem:hidden}.  
  %Equation (\ref{equ:rem_contribh}).
\end{proof}

We now study each controllable part. Define, for a separating center $P_l$:
  \begin{equation*}
    {\mathcal E}_l =\sum_{D_j \in {\mathcal D}_l}
    \sum_{P \in D_j} w(D_j) \kappa_P ({\mathcal F}_k, D_j)  
    - \delta(P_l) w(D(P_{l})) + 
    \sum_{\substack{D_j  \not \subset \mathrm{Inv}(E_k) \\
        D_j \in {\mathcal D}_l}}
    w(D_j)  (2 - v_{\overline{d}}^{a} (D_j)).
  \end{equation*}
  (the part of Hertling's formula corresponding to
  $\mathcal{D}_l$). Notice that we do not care if $l=0$ or not in this
  definition. The following result is the crux of this section:

\begin{lem}\label{lem:epsilon-l}
  % Let $P_l$ be a separating center and let $j$ be such that
  % $D_j=D(P_l)$. Set $\delta^0=1$ and $\delta^l=\delta_j(Q)$ where $Q$
  % is the only point in $D_j$ with $\pi_{l+1}\circ\cdots\pi_k(Q)=P_l$,
  % if $l>0$. Let
  With the notations above, $\mathcal{E}_l\geq w(D(P_l))$ and
  \begin{equation}
    \label{equ:el}
    {\mathcal E}_{l} \geq \frac{1}{2}
    \sum_{j=1}^{n_l} \nu_0 (\gamma^{l,j}) +
    \frac{1}{2} \sum_{j=1}^{m_l} \nu_0 (\overline{\gamma}^{l,j})    .
  \end{equation}
\end{lem}
\begin{proof}
  Denote by $P_j^l$ the intersection $P_j^l=\gamma^{l,j}_k\cap E$ (and
  the same for $\overline{P}_j^l$ and
  $\overline{\gamma}^{l,j}$). Since
  $(\pi_{l+1} \circ \hdots \circ \pi_{k}) (P_{j}^{l}) = P_l$ and the
  analogous property for $\overline{\gamma}_{k}^{l,r} \cap E$ holds, it
  follows that the multiplicities of the $\gamma^{l,j}$ and
  $\overline{\gamma}^{l,r}$ are integer multiples of $w(D(P_l))$ for
  all $1 \leq j \leq n_{l}$ and $1 \leq r \leq m_l$.

  The first center $P_0$ is special because it has no ancestor. Thus, the case $D_1$ non-invariant needs to be studied separately. In this case, the unique divisor in
  ${\mathcal D}_0$ is $D_1$ and  we get
$$   {\mathcal E}_{0} \geq \sum_{j=1}^{n_0} \nu_0 (\gamma^{0,j}) +1 \geq 
\frac{1}{2} \sum_{j=1}^{n_0} \nu_0 (\gamma^{0,j})  +1 \geq \frac{1}{2} \sum_{j=1}^{n_0} \nu_0 (\gamma^{0,j})  + \frac{1}{2} \sum_{j=1}^{m_0} \nu_0 (\overline{\gamma}^{0,j})  $$
by Equation (\ref{eq:non-inv-branches}), and this case is finished.

If $l >0$ or $D_1$ is invariant, then
$\overline{\gamma}^{l,1}, \hdots, \overline{\gamma}^{l,m_l}$ intersect
${\mathcal D}_l$ in $m_l$ non-invariant divisors
$D_{r_1}, \hdots, D_{r_{m_l}}$ that satisfy
$v_{\overline{d}}^{a} (D_{r_{\ell}}) \leq 1$ for
$1 \leq \ell \leq m_l$, since the irreducible components of
$\overline{\Gamma}$ are smooth,   by Remark \ref{rem:ancestor_valence}.   Using Equation \eqref{eq:non-inv-branches},
we deduce:
  \begin{equation}
\label{equ:ely}
\begin{split}
  {\mathcal E}_{l} \geq \sum_{j=1}^{n_l} \nu_0 (\gamma^{l,j}) +
  \sum_{\ell=1}^{m_l}  w(D(P_l))  (2 -v_{\overline{d}}^{a} (D_{r_{\ell}})) -
  w(D(P_l))  \geq   \\ 
  \sum_{j=1}^{n_l} \nu_0 (\gamma^{l,j}) +  \sum_{j=1}^{m_l} \nu_0 (\overline{\gamma}^{l,j})  - w(D(P_l))   .
\end{split}
\end{equation} 
Notice that $w(D(P_l))=1$ if $m_l \geq 1$ since $\overline{\Gamma}$ consists of smooth branches. The inequality (\ref{equ:el}) holds whenever
 \begin{equation}
 \label{equ:condition} 
 \sum_{j=1}^{n_l} \nu_0 (\gamma^{l,j}) +  \sum_{j=1}^{m_l} \nu_0 (\overline{\gamma}^{l,j})  \geq 2 w(D(P_l)).
\end{equation}
%and it is this latter result which we are going to prove.
 
Since the desingularization of $\Gamma \cup \overline{\Gamma}$ is
minimal, \eqref{equ:condition} holds if no descendant $Q$ of $P_l$ is
a separating center.  Assume, then, that $P_l$ does not satisfy
Property (\ref{equ:condition}) and hence $P_l$ has a descendant which
is a separating center. Since (\ref{equ:condition}) does not hold and
the left hand side is a multiple of $w(D(P_l))$, it follows that it is
less than or equal to $w(D(P_l))$. Thus, it suffices to show that
${\mathcal E}_l \geq w(D(P_l))$.

The hypothesis on $P_l$ implies that there exists a non-invariant
divisor $D_{\ell}$ in ${\mathcal D}_l$. This gives
 \begin{equation}
 \label{equ:el3}
  {\mathcal E}_{l} \geq \sum_{D_j \in {\mathcal D}_l} \sum_{P \in D_j} w(D_j) \kappa_P ({\mathcal F}_k, D_j)   + (2 - v_{\overline{d}}^{a} (D_{\ell}) - \delta(P_l))  w (D(P_l)) .
  \end{equation}
 Since ${\mathcal D}_l$ contains at least $v_{\overline{d}}^{a} (D_{\ell})$ elements of ${\mathcal I}$, then
$$   {\mathcal E}_{l} \geq  (2 -  \delta(P_l))  w (D^{l})   \geq w (D^{l})  $$
by Equation (\ref{equ:contribh}), and the result follows.
\end{proof}
The main result of this section is now straightforward,   it is Theorem \ref{teo:poincare-reducible-intro} that we restate here
for the sake of the reader.  
\begin{teo}\label{teo:poincare-reducible}  
Let $\hat{\Gamma}$ be a weakly isolated curve for ${\mathcal F}$, where ${\mathcal F}$ is a 
 germ of holomorphic
  foliation $\mathcal{F}$ defined in a neighbourhood of the origin in
  $\mathbb{C}^2$.
 % Let $\Gamma \cup \overline{\Gamma}$ be a curve that is invariant by a germ of holomorphic
 % foliation $\mathcal{F}$ defined in a neighbourhood of the origin in
 % $\mathbb{C}^2$.
 % Let $\pi :(M,E) \to ({\mathbb C}^{2},0)$ be the minimal desingularization of $\Gamma \cup \overline{\Gamma}$. Given an irreducible component $\gamma$ of $\Gamma$, 
 %let $P_{\gamma}$ the point at which the strict transform of $\gamma$ meets $E$, and $D_{\gamma}$ the irreducible component of $E$ to which $P_{\gamma}$ belongs.
  %Assume that the following conditions hold
  %\begin{itemize}
  %\item  $\kappa_{P_\gamma} (\pi^{*} {\mathcal F}, D_{\gamma}) \geq 1$  for any irreducible 
  %component $\gamma$ of $\Gamma$, and
  %\item $\overline{\Gamma}$ consists of zero, one or two (transversal, in this case) smooth invariant branches such 
 %that $\Gamma \cup \overline{\Gamma}$ is singular and $\Gamma \cap \overline{\Gamma} \subset  \{0\}$.   
  %\end{itemize}
    Then $\nu_0(\mathcal{F})\geq \nu_0(\hat{\Gamma})/2$.  
  %Then $\nu_0(\mathcal{F})\geq \nu_0(\hat{Gamma} \cup \overline{\Gamma})/2$.  
%  Moreover, $\nu_0(\mathcal{F})\geq \nu_0(\Gamma)/2 +1$ holds if $D_1$ is  non-invariant.
\end{teo} 
\begin{proof}
We have
 \begin{equation*}
%\label{equ:Hertling}
\begin{split}
  \nu_0 ({\mathcal F})  = \sum_{D_j} \sum_{P \in D_j}
  w(D_j) \kappa_P ({\mathcal F}_k, D_j)  
- \sum_{P \mathrm{\ s.c.}} \delta(P_l)w(D(P_l)) + \\ 
\sum_{D_j  \not \subset \mathrm{Inv}(E_k)}
w(D_j)  (2 - v_{\overline{d}}^{a} (D_j))
\end{split}
\end{equation*} 
by Hertling's formula \eqref{equ:Hertling} and Equation
\eqref{eq:dicritical-part-divided}. Since the union of the
${\mathcal D}_l$ is the set   of irreducible components of the exceptional divisor $E_k$ of $\pi$ and
$\mathcal{D}_l \cap \mathcal{D}_j$ is at most a point for $l \neq j$ and hence does not contain an irreducible component of $E_k$, we get 
\[ \nu_0 ({\mathcal F}) = \sum_{P_l \ \mathrm{s.c.}}   {\mathcal E}_{l} . \]
 We deduce
\[ \nu_0 ({\mathcal F}) = \sum_{P_l \ \mathrm{s.c.}}   {\mathcal E}_{l} \geq 
\frac{1}{2} \sum_{P_l \ \mathrm{s.c.}}  \left(    \sum_{j=1}^{n_l} \nu_0 (\gamma^{l,j}) +
 \sum_{j=1}^{m_l} \nu_0 (\overline{\gamma}^{l,j}) \right) = \frac{\nu_0(\hat{\Gamma})}{2} , \]
 where the inequality is a consequence of Lemma \ref{lem:epsilon-l}.
\end{proof}

We want to stress again that the structure of $\mathcal{F}$ along
$E_k$ is totally irrelevant except for the property that
$\hat{\Gamma}$ is weakly isolated, which only affects the intersection
points of $\hat{\Gamma}_k$ and $E_k$ along the resolution of
singularities of $\hat{\Gamma}$. The argument works whatever the
family of separating centers $P_l$ is and whatever non-invariant
irreducible components of $E_k$ are for $\mathcal{F}_k$, as long as
the weakly isolation holds.

The previous result provides a lower bound for the multiplicity of the foliation in terms 
of the multiplicity of an invariant curve but we do not require that all irreducible 
components are isolated. This will be very useful in desingularization settings in which 
invariant divisors can not be assumed to be isolated 
for the foliations ${\mathcal F}_1, \hdots, {\mathcal F}_k$.
% \begin{cor}\label{cor:poincare-reducible}
%  Let $\Gamma$ be a singular curve that is invariant by a germ of holomorphic
 % foliation $\mathcal{F}$ defined in a neighbourhood of the origin in
%  $\mathbb{C}^2$. 
% Assume that  $\Gamma$ is isolated. 
% Then $2\nu_0(\mathcal{F})\geq \nu_0(\Gamma)$.
% \end{cor} 
\begin{rem}
%Theorem \ref{teo:poincare-reducible} implies Theorem \ref{teo:poincare-reducible-intro}. 
Let us remark that 
in \cite{Corral-Fernandez:2006}, they proved that $M \nu_{0} ({\mathcal F}) \geq \nu_{0} (\Gamma)$ for some $M>0$
in the isolated case. We have shown that $M\leq 2$.
\end{rem} 
%Even more, the fact that we are assuming that $\hat{\Gamma}$ is weakly
%isolated forces its singular components to meet the exceptional
%divisor $E_k$ of its desingularization $\pi$ either in an invariant
%irreducible component or in a singular point of the pull-back of
%$\mathcal{F}$. Our approach uses this fact, together with Remark
%\ref{rem:hidden} by dividing $E_k$ into connected unions of
%irreducible components: each starting with the exceptional divisor
%corresponding to the blow-up of a trace point of a non-invariant
%component. In order to carry out this division, we require some
%nomenclature.

 \section{Global Poincar\'{e} problem}
 In the previous sections we studied the Poincar\'{e} problem in the local setting.
 We want to apply  Theorem \ref{teo:poincare-reducible}
 to obtain linear lower bounds for the multiplicity
 of a foliation in terms of the multiplicity of an invariant curve. 

We consider an algebraic curve $\Gamma$ in ${\bf C}{\bf P}(2)$ that is invariant 
by a foliation ${\mathcal F}$. 
Carnicer's \cite{Car:1994} solution of the Poincar\'{e} problem for the case where the curve $\Gamma$ 
does not contain dicritical singularities of ${\mathcal F}$ relies on showing the following local 
property: let ${\mathcal F}$ be a germ of non-dicritical foliation that preserves the curve $\Gamma$. 
Consider a reduced equation $f \in {\mathcal O}_2$ of $\Gamma$ and the foliation 
${\mathcal H}$ given by the first integral $f$. Then we always have
\[ Z_{P} ({\mathcal F}, \gamma) \geq Z_{P}({\mathcal H},\gamma) \] for
any $P \in \Gamma$ and any branch $\gamma$ of $\Gamma$ defined in a
neighborhood of $P$. In order to obtain lower bounds for
$\nu_0 ({\mathcal F})$ we need to bound
$Z_{P} ({\mathcal F}, \gamma) / Z_{P}({\mathcal H},\gamma)$ from
below.

Let us see one of the difficulties. Consider an irreducible curve $\Gamma$.
Suppose now that $\nu_0 (\Gamma)/ \mu (\Gamma) \leq M$ for some $M \in {\mathbb N}$.
In such a case the quotient $\nu_0 ({\mathcal F})/ \nu_0 ({\mathcal H})$ is bounded from below 
by a positive constant; indeed we have 
%Consider an irreducible curve $\Gamma$.
%Suppose now that $m(\Gamma)/m_{red}(\Gamma) \leq M$ for some $M \in {\mathbb N}$. 
%We have
\[ \frac{\nu_0 ({\mathcal F})}{\nu_{0}({\mathcal H})+1} = 
\frac{\nu_0 ({\mathcal F})}{\nu_0 (\Gamma)} = 
 \frac{\nu_0 ({\mathcal F})}{\mu(\Gamma)} \frac{\mu(\Gamma)}{\nu_0 (\Gamma)}  
 \geq  \frac{1}{M} \]
 by Equation (\ref{equ:in_virtual}). 
 Notice that $\nu_{0}({\mathcal H}) = \nu_0 (\Gamma) -1$ since ${\mathcal H}$ is a generalized
 curve  \cite{Camacho-Lins-Sad:1984}.
This   motivates us to study whether in such a case   
$Z_{P} ({\mathcal F}, \Gamma) / Z_{P}({\mathcal H},\Gamma)$ 
could be bounded from below by a positive constant.
Next we show that this is not the case.

\begin{example}
Consider the curve $\Gamma$ given by the equation $y^{2} = x^{p}$ where
$p \geq 3$ is an odd number. 
The map $\theta(t)=(t^2,t^p)$ is a Puiseux parametrization of $\Gamma$.
The multiplicities of $\Gamma$ are $\nu_0 (\Gamma)=2$ and 
$\mu(\Gamma)=1$. Consider the foliation ${\mathcal F}$ that has the first integral 
$y^{2}/x^{p}$. It is the foliation defined by the vector field
$X=2 x \frac{\partial}{\partial x} + p y \frac{\partial}{\partial y}$.  Since
\[ d \theta (t) \left(t \frac{\partial}{\partial t} \right) =  2 t^2  \frac{\partial}{\partial x} + 
p t^p\frac{\partial}{\partial y} = X(\theta(t)),
\]
we obtain $\theta^{*} X= t \partial/\partial t$ and $Z_{0} ({\mathcal F}, \Gamma)=1$.
The vector field $Y:= 2y  \frac{\partial}{\partial x} +  p x^{p-1}  \frac{\partial}{\partial y}$ 
is tangent to foliation as $d (y^{2} - x^{p})=0$. Since 
\[ d \theta (t) \left(t^{p-1} \frac{\partial}{\partial t} \right) =  2 t^p  \frac{\partial}{\partial x} + 
p t^{2p-2} \frac{\partial}{\partial y} = Y(\theta(t)),
\]
we get $Z_{0} ({\mathcal H}, \Gamma)= p-1$.
Hence, even if $\nu_0 (\Gamma)/\mu(\Gamma) = 2$ the quotient 
$Z_{0} ({\mathcal F}, \Gamma) / Z_{0} ({\mathcal H}, \Gamma)$ is
not bounded from below by a positive constant.
\end{example}

As a consequence of Theorem \ref{teo:half} the situation is different
for weakly isolated invariant curves.  Let us assume for now that
Theorem \ref{teo:half} holds and that the simpler Lemma
\ref{lem:lines} (see page \pageref{lem:lines}) does too, to prove Theorem \ref{teo:main}.
\begin{proof}[Proof of Theorem \ref{teo:main} assuming Theorem \ref{teo:half}  and Lemma \ref{lem:lines}]
We can consider that the line $L_\infty$ at $\infty$ is generic. 
In particular $L_\infty$ is not ${\mathcal F}$-invariant, $L_\infty \cap \mathrm{Sing} ({\mathcal F}) = \emptyset$ and 
$\Gamma$ intersects $L_\infty$ transversally.
We denote by $(x,y)$ the coordinates in the affine
chart ${\bf C}{\bf P}(2) \setminus L_\infty$.
Let $F$ be a polynomial vector field, with $\mathrm{cod} ({\mathrm Sing} (F)) \geq 2$, defining the foliation ${\mathcal F}$
in the affine chart ${\bf C}{\bf P}(2) \setminus L_\infty$.
We consider an irreducible equation $f \in {\mathbb C}[x,y]$ of the curve $\Gamma$ in 
${\bf C}{\bf P}(2) \setminus L_\infty$.
Let ${\mathcal H}$ be the foliation given by the hamiltonian vector field 
$H:=\frac{\partial f}{\partial y} \frac{\partial}{\partial x} - 
\frac{\partial f}{\partial x} \frac{\partial}{\partial y}$. 
%We denote by $Z_{\mathcal F}(P)$ (resp. $Z_H (P)$) the vanishing order of ${\mathcal F}$ 
%(resp. $H$) along $\Gamma$ at $P$.

Consider the normalization $\pi: \hat{\Gamma} \to \Gamma$. 
We lift $F$ and $H$ to the smooth compact Riemann surface  $\hat{\Gamma}$.
We denote by $c$ the number of irreducible components of $\Gamma$; it coincides with
the number of connected components of $\hat{\Gamma}$. We define
\[ Z_{P}({\mathcal F}) = Z_{\pi (P)} ({\mathcal F}, \pi (\hat{\Gamma}, P)) \ \ \mathrm{and} \ \ 
 Z_{P}({\mathcal H}) = Z_{\pi (P)} ({\mathcal H}, \pi (\hat{\Gamma}, P))  \]
 for $P \in \hat{\Gamma}$, where $ \pi (\hat{\Gamma}, P)$ is the germ of $ \pi (\hat{\Gamma})$
 at $P$. We claim that $Z_{Q} ({\mathcal F})\geq  Z_{Q} ( {\mathcal H}) /2$ for any 
$Q \in \hat{\Gamma}$. It is a consequence of $ Z_{Q} ({\mathcal H}) =0$ if 
$\pi (Q)$ is a regular point of $\Gamma$.
%Moreover, the inequality holds if  $\pi (Q)$ is a normal crossings singularity of $\Gamma$
%since $Z_{Q} ({\mathcal F}) \geq Z_{Q} ({\mathcal H}) =1$. 
In the singular case,  we apply 
Theorem \ref{teo:half}.

We apply Poincar\'{e}-Hopf's theorem to the restrictions of $F$ and $H$ to ${\Gamma}$. 
If we denote by $Z _F$ and $P_F$ the number of zeros and poles (with multiplicity) of 
$F_{|\Gamma}$, we obtain
\[ Z_F - P_F = Z_H - P_H =\chi (\hat{\Gamma}) \] 
where $\chi$ stands for the Euler characteristic. 
We have $Z_{F} \geq  Z_{H}/2$ by the previous discussion. Moreover, 
it is well-known that $P_H= m(m-3)$ and 
$P_F= m (d-1)$ where $m= \deg (\Gamma)$ and $d = \deg ({\mathcal F})$
\cite[Proposition 25.22]{Ilya-Yako:2008}.
%(see the Ilyashenko-Yakovenko book).  \marginpar{reference}
 We have
\[ m (d-1)  = P_H + Z_F - Z_H \geq  P_H - \frac{Z_H}{2} = 
 \frac{P_H}{2} - \frac{\chi (\hat{\Gamma})}{2}  
 \geq \frac{m(m-3)}{2} - c \]
and then 
\begin{equation}
\label{equ:aux}
 d \geq  \frac{m-3}{2} - \frac{c}{m} +1 = \frac{m}{2} - \frac{c}{m} - \frac{1}{2}. 
\end{equation} 
Assume $c < m$. We deduce
\[   m \leq  2 d +  \frac{2c}{m} + 1  <  2d+3 . \]
It follows
that $m \leq 2d+2$. 

Consider the remaining case  $c=m$. Thus, all irreducible components of $\Gamma$ have
degree $1$, i.e. they are lines. We have
$Z_{F} \geq Z_H$ by Lemma \ref{lem:lines} and hence 
\[ m (d-1)  = P_H + Z_F - Z_H \geq  P_H   = m(m-3) \implies m \leq d+2. \]
Therefore we get $m \leq d+2 \leq 2d +2$. 

Finally, assume that $\Gamma$ is irreducible. 
Since $d=0$ implies $m=1$, we can assume $d \geq 1$. Therefore $m \leq 2d+1$ holds 
if $m \leq 3$.  So, it suffices to consider $m \geq 4$. Since 
\[   m \leq  2 d +  \frac{2c}{m} + 1   = 2 d +  \frac{2}{m} + 1 <  2d+2, \]
it follows that $m \leq 2d+1$. 
\end{proof}
\begin{rem}
Consider a foliation ${\mathcal F}$ of ${\bf C}{\bf P}(2)$ of degree $0$
and let $P$ be its unique singular point. Notice that the unique invariant curves $\Gamma$ that 
satisfy our hypothesis are either a line through $P$ or two lines through $P$. 
In the former case we have $\deg (\Gamma)=1 = 2 \deg ({\mathcal F}) + 1$ whereas in the 
latter case we obtain $\deg (\Gamma)=2 = 2 \deg ({\mathcal F}) + 2$.
\end{rem}
\subsection{Comparison of vanishing orders}
We show Theorem \ref{teo:half} in the remaining of the paper. Let us 
assume that ${\mathcal F}$ is a germ of foliation defined in a neighborhood of $(0,0)$
in ${\mathbb C}^{2}$ without lack of generality.  
We desingularize  
$\Gamma$ along an irreducible component $\gamma$ of $\Gamma$.
Consider the notations in Definition \ref{def:not}. In this case 
$\pi_1, \pi_2, \hdots, \pi_{k}$ is a sequence of blow-ups of infinitely near points 
of  $\gamma$. 
%where $\Gamma_0 = \Gamma$, $\gamma_0 = \gamma$, 
%$P_0 =(0,0)$ and the remaining notations in Definition \ref{def:not}.
\begin{defi}
We say that $\pi$ is a  
desingularization of  $\Gamma$ along $\gamma$  if 
$\gamma_k$ intersects the divisor transversally at the non-corner point $P_k$   of $E_k$   and 
the germs of $\gamma_k$ and $\Gamma_k$ at $P_k$ coincide.
We assume that $\pi$ is minimal with such a property.
\end{defi}
%Notice that  if $\gamma$ is smooth but $\Gamma \neq \gamma$ we have to blow-up 
%infinitely near points of $\gamma$ until we separate $\gamma$ from the other 
%irreducible components of $\Gamma$.  
\begin{rem}
The previous property is not equivalent to a desingularization of $\gamma$. 
For instance if $\gamma$ is smooth and $\Gamma$ is not, we need to blow-up the origin 
since $\Gamma_0 \neq \gamma_0$.
\end{rem}

By appying iteratively Equations (\ref{equ:znon-dic}) and (\ref{equ:zdic}), we obtain 
\begin{equation}
\label{equ:ziter}
 Z_{0} ({\mathcal F}, \gamma) = \nu^{\gamma}_0 \tau_0 + \hdots + \nu^{\gamma}_{j-1}\tau_{j-1} + 
Z_{P_j} ({\mathcal F}_j,  \gamma_j) 
\end{equation}
for any $0 \leq j \leq k$, 
where $\nu^{\gamma}_j=\nu_{P_j} (\gamma_j)$  and $\tau_j = \nu_{P_j} ({\mathcal F}_j)$ 
%is the multiplicity of $\gamma_j$ at $P_j$ and $\tau_j = \nu_{P_j} ({\mathcal F}_j)$ 
if ${\mathcal F}_j$ is $1$-dicritical at $P_j$ and 
$\tau_j = \nu_{P_j} ({\mathcal F}_j)-1$  otherwise.
%This formula can be found in the 
%Ilyashenko-Yakovenko book.  \marginpar{reference}
Since ${\mathcal H}$ is non-dicritical, we get
\[ Z_{0} ({\mathcal H}, \gamma) = \nu^{\gamma}_0(\nu_{P_0} ({\mathcal H}_0)-1) + \hdots + 
\nu^{\gamma}_{j-1}(\nu_{P_{j-1}} ({\mathcal H}_{j-1})-1) +  Z_{P_j} ({\mathcal H}_j, \gamma_j) \] 
for any $0 \leq j \leq k$.
\begin{rem}
\label{rem:mh}
The germ of ${\mathcal H}_j$ at any point is a generalized curve;
therefore we get
\[ \nu_{P_{j}} ({\mathcal H}_{j}) = \nu_{P_{j}} (\tilde{\pi}_j^{-1} (\Gamma)) - 1 = \nu_{P_j} (\Gamma_{j}) + m_j -1 , \]
where $m_j$ is the number of irreducible components of $E_j$ containing $P_j$ \cite{Camacho-Lins-Sad:1984}.  
Notice that $m_j=0$ if $j=0$. We obtain
\begin{equation*}
  Z_{0} ({\mathcal H}, \gamma) =
 \sum_{l=0}^{j-1} 
  \nu^{\gamma}_{l}(\nu_{P_{l}} (\tilde{\pi}_{l}^{-1} (\Gamma) )-2)
  +  Z_{P_j} ({\mathcal H}_j, \gamma_j)
\end{equation*}
for any $0 \leq j \leq k$, where $\tilde{\pi}_{0}^{-1}(\Gamma)=\Gamma$ by convention.  
Notice that $\nu_{P_{l}} (\tilde{\pi}_{l}^{-1} (\Gamma)) - 2 \leq \nu_{P_{l}}(\Gamma_{l})$
for $0 \leq l \leq k$.    
\end{rem} 
\begin{rem}
Given $0 \leq j < k$, we have $ \tau_{j} \geq \nu_{P_j} ({\mathcal F}_{j}) -1$. 
%\[ \tau_{j} \geq \nu_{P_j} ({\mathcal F}_{j}) -1  \geq \frac{1}{2} \nu_{P_j} (\Gamma_{j}) -1 . \]
As  we are going to use Equation (\ref{equ:ziter}) to obtain lower bounds of $Z_{0} ({\mathcal F}, \gamma)$, 
  we want to consider points $P_j$ that have a non-negative contribution  $\tau_{j}$ to 
Equation (\ref{equ:ziter}).  Indeed we will consider points $P_j$ with $\nu_{P_j} ({\mathcal F}_{j}) \geq 1$.
This motivates the next definition.  
\end{rem}
\begin{defi}  
\label{def:i}
Let $I= \{  0, 1, \hdots, k-1 \}$ be the set of indices of blow-up centers, and consider those where $\Gamma_l$ has multiplicity $1$:
$$I_{1}= \{ l \in I: \nu_{P_{l}} (\Gamma_l)=1 \}.$$   
Define $\iota$ as the maximum  element of $I$ with $\nu_{\iota}^{\gamma}=\nu_{P_{\iota}}(\gamma_{\iota})>1$ (if $I_1=\emptyset$ then $\iota$ is irrelevant and can be defined as $-1$). We set   $\Omega_1=I_1$ if $\iota\geq 0$ and $D_{\iota +1}$ is non-invariant, and $\Omega_1=\emptyset$ otherwise.
We define   $\rho= \min (I_1 \cup \{k\} ) -1$  (that is, the last index such that  $\nu_{P_{\rho}}(\Gamma_{\rho} )>1$).  
% We define $\Omega_1 = \emptyset$ if $I_1 = \emptyset$.  Otherwise, let  $\iota$ be the last index such that 
% $\nu_{\iota}^{\gamma}=\nu_{P_{\iota}} (\gamma_{\iota}) > 1$. 
% We define  $\Omega_1 = \emptyset$ if $D_{\iota +1}$ is invariant and $\Omega_1 = I_1$ otherwise.
\end{defi}
\begin{rem}
  The set $I_1$ is the ``final stage'' in the resolution of singularities of both $\gamma$ and $\Gamma$: for $j\in I_1$, 
 the germs $(\Gamma_j, P_j)$ and $(\gamma_j, P_j)$   in $\tilde{\pi}_{j}^{-1} ({\mathbb C}^{2})$   coincide, 
  %$\Gamma_j=\gamma_j$, 
  $P_{j}$ is of multiplicity $1$ (for both of them, obviously), and it is also the corner $D_j\cap D_{\iota+1}$. 
  The set $\Omega_1$ is non-empty if and only if $D_{\iota+1}$ is non-invariant. In this case, for $j\in \Omega_1=I_1$, 
  $P_j$ always belongs to at least one non-invariant component of $E_j$.

Also, $I_1\neq \emptyset$ implies that $\gamma$ is a singular curve 
(otherwise, as soon as $\nu_{P_j}(\Gamma_j)=1$, we should have $j=k$, so that $I_1=\emptyset$).
\end{rem}
\begin{rem}
\label{rem:mult_1}
We have $\nu_{P_j} ({\mathcal F}_{j}) \geq 1$ if $j \in I \setminus \Omega_1$. 
This is clear if $j \in I \setminus I_1$ since  $\nu_{P_{j}} (\Gamma_j) \geq 2$.
Moreover, it also holds if $j \in I_1 \setminus \Omega_1$  since  
there are two invariant curves   in $\tilde{\pi}_{j}^{-1} ({\mathbb C}^{2})$,    namely $\Gamma_{j}$ and $D_{\iota +1}$, containing $P_j$.
This is what makes $\Omega_1$ so important: it contains the ``worst'' centers in terms of lower bounds for $\nu_{P_j}(\mathcal{F}_j)$; this will become clear as we proceed.
\end{rem}
Before continuing, notice that the inequality we wish to prove can be written, by  Remark \ref{rem:mh}, as
\begin{equation}\label{eq:big-inequality}
  \frac{\nu^{\gamma}_0\tau_0 + \hdots + 
    \nu^{\gamma}_{\rho}\tau_{\rho}
    +  Z_{P_{\rho+1}} ({\mathcal F}_{\rho+1}, \gamma_{\rho +1})}
  {    \sum_{j=0}^{k-1} 
  \nu^{\gamma}_{j}(\nu_{P_{j}} (\tilde{\pi}_{j}^{-1} (\Gamma) )-2)
  +  Z_{P_{k}} ({\mathcal H}_{k}, \gamma_{k})} \geq \frac{1}{2}.
\end{equation}   
We are going to partition both numerator and denominator of the left
hand side and verify that, for each of the sets of the partition, the
corresponding sums are both positive and satisfy the inequality, and
that will finish the argument.

First of  all, let us consider the last terms   of the case $I_1 = \emptyset$.  
\begin{lem}
\label{lem:fin}
$Z_{P_k} ({\mathcal H}_k, \gamma_k) = 1$, and $Z_{P_k} ({\mathcal F}_k,  \gamma_k)  \geq 1$ if $D_k$ is invariant.
\end{lem} 
\begin{proof}
The pair $({\mathcal H}_k, \gamma_k)$ is analytically conjugated in a neighborhood 
of $P_k$ to the pair $(d(x^p y^q)=0,y=0)$ where $p,q \geq 1$. 
We obtain  $Z_{P_k} ({\mathcal H}_k, \gamma_k) = 1$. If $D_{k}$ is invariant, then 
$Z_{P_k} ({\mathcal F}_k,  \gamma_k)  \geq 1$ since $P_k$ is a singular point.
%The second item is a consequence of the isolated nature of $\gamma$.
\end{proof}
%\begin{rem}
%\label{rem:weak_isolation}
%As a matter of fact, we will prove in what follows a stronger version of Theorem \ref{teo:half}: we only require
% $\kappa_{P_\gamma} (\pi^{*} {\mathcal F}, D_{\gamma}) \geq 1$ for any irreducible 
%component $\gamma$ of $\Gamma$   except at most $2$, where the exceptions are smooth and transverse (when there are two of them), and   
%$\pi :(M,E) \to ({\mathbb C}^{2},0)$  is the minimal desingularization of $\Gamma$ (see Theorem \ref{teo:poincare-reducible}). 
%This hypothesis is much weaker than the   weak   isolation of $\Gamma$: we only rule out very specific families of invariant equisingular curves.
%\end{rem}
The following result was used in the proof of Theorem \ref{teo:main} to improve an inequality.  
\begin{lem}
\label{lem:lines}
Let $\Gamma$ be a weakly isolated curve composed of pairwise transverse smooth branches.
Then $Z_{0} ({\mathcal F},\gamma) \geq  Z_{0} ({\mathcal H},\gamma) \geq 1$ for any irreducible 
component $\gamma$ of $\Gamma$.
\end{lem}
\begin{proof}
We have
\[  Z_{0} ({\mathcal H},\gamma) = (\nu_{0} ({\mathcal H})-1) + 
Z_{P_1} ({\mathcal H}_1, \gamma_1) =  ( \nu_0 (\Gamma) -2) +1 = \nu_0 (\Gamma) -1.  \]
Assume $D_1$ is invariant. We have
$Z_{P_1} ({\mathcal F}_1, \gamma_1)  \geq 1$ and
\[  Z_{0} ({\mathcal F},\gamma) = \tau_{0} + Z_{P_1} ({\mathcal F}_1, \gamma_1)  \geq 
 \nu_0 ({\mathcal F}) =  \sum_{P \in D_1} Z_{P} ({\mathcal F}, D_1)  -1  \geq \nu_0 (\Gamma)-1,  \]  
since  $Z_{P_{\gamma'}} ({\mathcal F}, D_1) \geq 1$ for any irreducible component $\gamma'$ of $\Gamma$.

Assume $D_1$ is non-invariant. We have
 \[ Z_{0} ({\mathcal F},\gamma) \geq \tau_0 = \nu_0 ({\mathcal F}) = 1 + 
 \sum_{P \in D_1} \mathrm{tang}_{P} ({\mathcal F}, D_1) \geq 1 + (\nu_0 (\Gamma) -2) \]
by   the weak isolation   hypothesis. In any case,
$Z_{0} ({\mathcal F},\gamma)  \geq \nu_0 (\Gamma) - 1 = 
Z_0 ({\mathcal H}, \gamma)$.
\end{proof}
\begin{defi}
Let $0 \leq j \leq k$. From now on, $F_j$ will denote the union of the (at most two, obviously) irreducible components of $E_j$ containing $P_j$,  and $F_{j}'$ its subset of invariant irreducible components. We denote by  $m_j$ be the number of irreducible components of $F_j$.
\end{defi}
\begin{rem} 
\label{rem:blow_isolation}
We are going to apply Theorem \ref{teo:poincare-reducible} at $P_j$ for $j \in I \setminus \Omega_1$ to obtain lower bounds for $\tau_{j}$.
This approach works since weak isolation  is invariant by blow-ups by Proposition \ref{pro:weak_invariance}.
%
%
%  Notice that the   
%desingularization of $\Gamma_j \cup F_{j}'$ provided by $\pi_{j+1} \circ \hdots \circ \pi_k$ and its minimal desingularization $\hat{\pi}_j$ do not coincide in general. 
 %  Anyway, we claim that    $\kappa_{P_{\gamma'}} (\hat{\pi}_{j}^{*} {\mathcal F}, D_{\gamma'}) \geq 1$ for any irreducible 
%component $\gamma'$ of $\Gamma_j \cup F_{j}'$ except (maybe)  for one or two (transverse) smooth components.    
%It holds for $j=0$ by hypothesis, so we can assume $j \geq 1$.
%Denote by $a_j$ the number of  irreducible components $\gamma'$ that are exceptional (i.e. $\kappa_{P_{\gamma'}} (\pi^{*} {\mathcal F}, D_{\gamma'}) = 0$)
%and pass through $P_j$. 
%We have $a_j \leq 1$ and $a_j + m_j \leq 2$  since exceptional components of $\Gamma$ are smooth (and transverse) and $j \geq 1$.
%The desingularization of   of $\Gamma_j \cup F_{j}'$ provided by $\pi_{j+1} \circ \hdots \circ \pi_k$
%is obtained from  $\hat{\pi}_j$ by doing further blow-ups over the points in which the strict transform of $F_j$ intersects $\hat{\pi}_{j}^{-1}(P_j)$
 %{}(e.g. $\mathcal{F}_j$ may have a hyperbolic singularity at $P_j$ whose two transverse separatrices are $\Gamma_j$, one of them is tangent to $F_j$, which has a 
 %single component, and $F^{\prime}_{j}=\emptyset$).  Since $a_j + m_j \leq 2$, 
 %we can apply Theorem  \ref{teo:poincare-reducible} to $\Gamma_j \cup F_{j}'$. 
% Hence we obtain $\nu_{P_j} ({\mathcal F}_{j}) \geq \nu_{P_j} (\Gamma_j \cup F_{j}') /2$.   
\end{rem}
The next lemma measures the contribution of each term to the required inequality whenever $j \not \in \Omega_1$.
\begin{lem}
 \label{lem:half}
  Let    $j \in I \setminus \Omega_1$.  
  %$0\leq j <k$ with $\nu_{P_j}(\Gamma_j)\geq 2$. 
  Let 
  \[ \theta_j = \tau_{j} - \frac{1}{2} (\nu_{P_j} ({\mathcal H}_j)-1) . \]
  Then $\theta_{j} \geq -1$. Moreover, the following non-exclusive statements hold:
  \begin{itemize}
  \item If $F_{j} = F_{j}'$ then $\theta_{j} \geq 0$ (this contains the case $j=0$).
  \item If $\mathcal{F}_j$ is $1$-dicritical at $P_j$ then $\theta_{j} \geq 0$. 
  \item If  $F_{j} = F_{j}'$ and $\mathcal{F}_j$ is $1$-dicritical at $P_j$ then $\theta_{j} \geq 1$. 
  \end{itemize}
%  \begin{equation*}
 %   \frac{\tau_j}{\nu_{P_j} ({\mathcal H}_j)-1} \geq \frac{1}{2} 
 % \end{equation*}
  \end{lem}
\begin{proof}
 We have
$\nu_{j} ({\mathcal H}_j)-1 = \nu_{P_j} (\Gamma_j)+ m_j -2$
by Remark  \ref{rem:mh}. Since 
\[ \tau_j \geq  \nu_{P_j} ({\mathcal F}_j)  -1 \geq \frac{1}{2} \nu_{P_j} (\Gamma_j) -1 \]
by  Theorem \ref{teo:poincare-reducible}, Remarks \ref{rem:mult_1} and \ref{rem:blow_isolation} and $\nu_{P_j} (\Gamma_j)  \geq \nu_{P_j} ({\mathcal H}_j)-1$, we obtain 
$\theta_{j} \geq -1$. Now we prove each of the statements.
  \begin{itemize}
  \item  If  $F_{j} = F_{j}'$,
  then
\[ \tau_j \geq \frac{1}{2} (\nu_{P_j} (\Gamma_j) +m_j) - 1 \geq 
\frac{1}{2} (\nu_{j} ({\mathcal H}_j)-1 )   \]
by Theorem  \ref{teo:poincare-reducible} and Remarks \ref{rem:mult_1} and  \ref{rem:blow_isolation}, from which $\theta_j\geq 0$ follows (the case $j=0$ is covered because $\Gamma$ is singular at $(0,0)$, which implies that $\nu_j(\mathcal{H}_j)\geq 1$).
\item If ${\mathcal F}_j$ is $1$-dicritical at $P_j$, then
$ \tau_j = \nu_{P_j} ({\mathcal F}_j) \geq    \frac{1}{2} \nu_{P_j} (\Gamma_j)$ 
by Theorem \ref{teo:poincare-reducible} and Remark \ref{rem:blow_isolation} and  $\theta_{j} \geq 0$ follows from
$\nu_{P_j} (\Gamma_j)  \geq \nu_{P_j} ({\mathcal H}_j)-1$.
%since the non-invariant divisor $D_{j+1}$ appears by blowing-up the non-corner point 
%$P_j$ (for the foliation ${\mathcal F}_j$), we deduce that 
%$\nu_{P_j} ({\mathcal F}_j) \geq \frac{1}{2} \nu_{P_j} (\Gamma_j) + 1$ 
%by Theorem \ref{teo:poincare-reducible}. We get
%\[ 
%\frac{\tau_j}{\nu_{P_j} ({\mathcal H}_j)-1}
%= \frac{\tau_j}{\nu_{P_j} (\tilde{\pi}_{j}^{-1}(\Gamma))-2} 
%\geq \frac{\nu_{P_j} ({\mathcal F}_j)}
%{\nu_{P_j} (\Gamma_j)}  
% \geq
%\frac{  \frac{1}{2} \nu_{P_j} (\Gamma_j)+1  }{\nu_{P_j} (\Gamma_j)} 
%  >    \frac{1}{2}.  \]
\item Finally, if $F_{j} = F_{j}'$ and $\mathcal{F}_j$ is $1$-dicritical at $P_j$ then we obtain 
\[  \tau_j = \nu_{P_j} ({\mathcal F}_j) \geq    \frac{1}{2} (\nu_{P_j} (\Gamma_j) +m_j)
=   \frac{1}{2} (\nu_{P_j} ({\mathcal H}_j)-1)  + 1  \]
by Theorem \ref{teo:poincare-reducible} and Remark \ref{rem:blow_isolation} and  hence $\theta_{j} \geq 1$.
\end{itemize}
  The proof is complete.
% or
%${\mathcal F}_{j-1}$ is non-$1$-dicritical at $P_{j-1}$.
%\end{itemize}
% Then 
% \[ \frac{\tau_j}{\nu_{P_j} ({\mathcal H}_j)-1} \geq \frac{1}{2} . \]
% \end{lem}
\end{proof}
\begin{defi}  
Given a subset $S$  of $I \setminus \Omega_1$, we define
 \[ \Theta_{S} =  \sum_{m \in S} \nu^{\gamma}_{m}\tau_{m}   - \frac{1}{2}  \sum_{m \in S} \nu^{\gamma}_{m}(\nu_{P_m} ({\mathcal H}_m)-1) .   \]  
\end{defi}
We now divide the sequence $(P_j)_{j \in I \setminus \Omega_1}$ into ``satisfactory'' subsequences. Assume $I \setminus \Omega_1=\bigcup_{i=1}^{\ell}S_i$
with $S_{i}\cap S_j=\emptyset$ if $i\neq j$. We will estimate $ \Theta_{S_i}$
%\begin{equation*}
 %\Theta_{S_i} =  \sum_{l\in S_{i}}\nu_l^{\gamma} \tau_l - \frac{1}{2} \sum_{l\in S_i}\nu_l^{\gamma}(\nu_{P_l}(\tilde{\pi}^{-1}(\Gamma))-2)
%\end{equation*} 
for each $i\in \left\{ 1,\ldots,\ell \right\}$  and these estimates will essentially imply Theorem \ref{teo:half}.   
The following definitions provide the required partition.  
%Of course, $\Omega_1$ will be one of the sets, but it will require a specific argument (Case 2 in \ref{case:omega-1}). 
\begin{defi}
\label{def:leader}
Let $0 \leq j \leq k$. We say that $P_j$ is a   {\it precursor point}  
if $j \in I \setminus \Omega_1$ and at least one of the following non-exclusive properties holds:
\begin{itemize}
\item $D_{j+1}$ is non-invariant and $\nu^{\gamma}_{j+1} < \nu^{\gamma}_j$;
\item $D_{j+1}$ is non-invariant and it is the unique non-invariant divisor containing $P_{j+1}$ in $E_{j+1}$;
\item every irreducible component of $E_j$ containing $P_j$ is invariant.
\end{itemize}
We say that $P_j$ is a {\it leader point} if it is a precursor point such that $F_j =F_{j}'$.  
\end{defi}
A   \emph{precursor point}   $P_{j}$ either only belongs to invariant divisors or (non-exclusively) marks the start of a chain $P_{j},\ldots, P_{r}$ of blow-ups such that $D_{j+1}$ is non-invariant, and for $l\in \{j+1, \ldots, r\}$, the curve $\gamma_l$  intersects   $D_{j+1}$ 
  or more precisely, its strict transform by $\pi_{j+2} \circ \hdots \circ \pi_{l}$ if $l > j+1$  
(an important well-known consequence is that the multiplicity of $\gamma_l$ at $P_l$ is constant for $l\in \left\{ j+1,\ldots,r-1 \right\}$). 
\begin{defi}
\label{def:partition}
We define a partition $\mathcal{P}'$ of the set $I \setminus \Omega_1$
as follows: a set $\{j, \hdots , r\}$ belongs to $\mathcal{P}'$ if
$P_{j}$ is a leader, $P_{j+1}, \hdots, P_{r}$ are non-leaders and
either $P_{r+1}$ is a leader or $r+1 \not \in I \setminus
\Omega_1$. Replacing ``leader" with ``precursor", we obtain another
partition ${\mathcal P}$ that is finer than ${\mathcal P}'$.
\end{defi}
Notice that if $\iota>-1$ (cf. Definition \ref{def:i}) 
  and $\Omega_1 \neq \emptyset$    
then $P_{\iota}$ is always a  precursor. Finally, \emph{roughly speaking}, the sets in $\mathcal{P}$ are usually singletons, unless there are specific chains of dicritical divisors.

As pointed out above,  the sets of the partition $\mathcal{P}$ are all well-behaved with respect to \eqref{eq:big-inequality}.
\begin{lem}\label{lem:half_c}
Let $S=\{j , \hdots, r\}$ be a set of the partition ${\mathcal P}$.    Set $\delta = 1$ if $F_{j}=F_{j}'$ and $\delta =0$ otherwise.
Then we have
\begin{itemize}
\item $\Theta_{S} \geq (\delta - 1)  \nu_{j}^{\gamma} $ in any case;
\item $\Theta_{S} \geq  \nu_{r+1}^{\gamma} + (\delta -1) \nu_{j}^{\gamma} $ if $D_{j+1}$ is non-invariant and $F_{r+1} \neq F_{r+1}'$;
\item $\Theta_{S} \geq 0$ if $D_{j+1}$ is invariant.
\end{itemize}
\end{lem}
\begin{proof}
We consider the two alternatives: $D_{j+1}$ non-invariant or invariant.   
As usual we identify $D_{j+1}$ with its strict transforms for the sake of simplicity. 
So for instance, when we say $P_{s} \not \in D_{j+1}$ for some $s > j+1$, we mean that $P_s$ is 
not in the strict transform of $D_{j+1}$ by $\pi_{j+2} \circ \hdots \circ \pi_{s}$.

\emph{Case $D_{j+1}$ non-invariant.}
Consider the sequence $P_{j+1}, \hdots, P_{s}$ of infinitely near points of $\gamma$ that belong to 
$D_{j+1}$. 
We claim that $r \leq s$. Assume $r \geq s+1$, aiming at contradiction. 
This implies that $P_s$  and $P_{s+1}$ are not precursors, by definition   of the partition ${\mathcal P}$.    We distinguish two cases:
\begin{itemize}
\item If $D_{s}$ is invariant or $s=j+1$,
then  $D_{j+1}$ is the unique non-invariant divisor containing $P_{s}$ (in the latter case it is a consequence of
$\nu_{j}^{\gamma} =  \nu_{j+1}^{\gamma}$  and the definition of precursor point).
Now, if  $D_{s+1}$ is non-invariant then 
$P_s$ is a precursor  since $P_{s+1} \not \in D_{j+1}$ (contradiction). Otherwise, if $D_{s+1}$ is invariant then $P_{s+1}$
is a precursor, 
%or $s+1 \not \in I_1$, 
providing also a contradiction.

\item If, on the contrary, $D_{s}$ is non-invariant and $s > j+1$,
then, since $P_{s-1}$ is not a  precursor point, we obtain $\nu_{s-1}^{\gamma} = \nu_{s}^{\gamma}$ and as a consequence 
$P_{s+1}$ does not belong neither to $D_{j+1}$ nor to $D_{s}$. We obtain a contradiction since  $P_{s}$ is a precursor if 
$D_{s+1}$ is non-invariant and $P_{s+1}$ is a precursor otherwise.
\end{itemize}
 
The equality $ \nu_{j}^{ \gamma} = \sum_{l=j+1}^{s} \nu_{l}^{\gamma}$ is a direct consequence of the fact that $P_{j+1},\ldots, P_s$ belong to $D_{j+1}$ and $P_{s+1}$ does not.
Now,  we obtain $\Theta_{S} \geq \delta  \nu_{j}^{\gamma} - \sum_{l=j+1}^{r} \nu_{l}^{\gamma}$ by Lemma \ref{lem:half} and hence
\[ \Theta_{S} \geq \delta  \nu_{j}^{\gamma}  - \sum_{l=j+1}^{r} \nu_{l}^{\gamma} \geq  \delta  \nu_{j}^{\gamma}  - \sum_{l=j+1}^{s} \nu_{l}^{\gamma}  = 
(\delta - 1) \nu_{j}^{\gamma} , \]
using that $r\leq s$ and Lemma \ref{lem:half}. There are two subcases to consider.

If $r < s$
or $D_{l+1}$ is non-invariant for some $j +1 \leq l \leq r$, then the inequality $\Theta_{S} \geq  \nu_{r+1}^{\gamma} + (\delta - 1) \nu_{j}^{\gamma}$   follows by  Lemma \ref{lem:half}.

On the other hand, if $r=s$ and $D_{l+1}$ is invariant for all $j +1 \leq l \leq r$, we obtain $F_{r+1} = F_{r+1}'$, which concludes this case.
%by considering a Puiseux parametrization of $\gamma_{j}$.
%This  equality together with $r \leq s$ and Lemma  \ref{lem:half}  imply Equation (\ref{equ:auxm4}).

\emph{Case $D_{j+1}$ invariant}. This implies $F_{j} = F_{j}'$ by Definition \ref{def:leader}
and either    $j+1 \not \in I \setminus \Omega_1$   or $P_{j+1}$ is a precursor and hence $S=\{j\}$   in both cases.  
The result is a consequence of Lemma \ref{lem:half}.
\end{proof}
\begin{lem} \label{lem:theta}
Let $S=\{j , \hdots, r\}$ be a set of the partition ${\mathcal P}'$ which is the union of consecutive sets $S_1, \hdots, S_m$ of ${\mathcal P}$.
Then the following inequalities hold:
\begin{equation}\label{eq:sums-between-initials}
 \Theta_{S}=  \sum_{l=1}^{m} \Theta_{S_l} \geq 0,\;\mathrm{and}\; \sum_{l=1}^{p} \Theta_{S_l} \geq {}  \nu_{1 + \max S_{p}}^{\gamma} {}  
\end{equation}
if $1\leq p < m$ or if $p=m$ and
$F_{1 + r}  \neq F_{1 + r}'$ (the latter condition can only happen if $r = \max (I \setminus \Omega_1)$).
In particular, we obtain $\Theta_{S} \geq  \nu_{1 + r}^{\gamma}$ if $F_{1 + r}  \neq F_{1 + r}'$.
\end{lem}
\begin{proof}
  There are two cases, depending on wether $D_{j+1}$ is invariant or not.

If $D_{j+1}$ is invariant, then, by definition of leader point $m=1$, $S_1 = \{ j \}$ and, by Lemma \ref{lem:half_c}, $\Theta_{S_1} \geq 0$. The second part holds because
  the condition    is empty (there is no $1\leq p< m$ and  $F_{1 + r}  = F_{1 + r}'$).

Assume that  $D_{j+1}$ is non-invariant. Denote $\delta = 1$ if 
$F_{1 + r}  \neq F_{1 + r}'$ and $\delta =0$ otherwise.
If $m=1$, then $\Theta_{S_1} \geq  \delta \nu_{1 + r}^{\gamma} $ straightforwardly by Lemma \ref{lem:half_c}. 
Suppose, then, that $m >1$. Then Lemma \ref{lem:half_c} implies the inequalities: $\Theta_{S_1} \geq \nu_{1 + \max S_1}^{\gamma}$, 
\[ \Theta_{S_l} \geq - \nu_{\min S_l}^{\gamma} + \nu_{1+ \max S_{l} }^{\gamma} \] 
for any $1 < l < m$ and  $\Theta_{S_m} \geq -  \nu_{\min S_m}^{\gamma} + \delta \nu_{1 + \max S_m}^{\gamma} $.   
A telescopic argument concludes the proof of the claim.
\end{proof}
At this point, we have all the machinery required to prove Theorem \ref{teo:half}.
\subsection{Proof of Theorem \ref{teo:half}}
%Denote
%$\rho= \min (I_1) -1$ 
 %(that is, the last index such that  $\nu_{P_{\rho}}(\Gamma_{\rho} )>1$): 
 By definition,   we have   $\rho \geq \iota$,  as $\gamma$ is a branch of $\Gamma$. 
 We are going to compare 
 \[ Z_{0} ({\mathcal F}, \gamma)= \sum_{l=0}^{\rho} \nu^{\gamma}_l \tau_l   +
Z_{P_{\rho+1}}({\mathcal F}_{\rho+1}, \gamma_{\rho+1}) \]
with 
\[ Z_{0} ({\mathcal H}, \gamma) = \sum_{j=0}^{k-1} \nu^{\gamma}_j(\nu_{P_j} ({\mathcal H}_j)-1)  + 
Z_{P_k} ({\mathcal H}_k, \gamma_k) \]
  to get the inequality (see Equations (\ref{equ:znon-dic}) and (\ref{equ:zdic})).  
%Recall that $I_1= \{ 0 \leq j \leq k-1 : \nu_{P_{l}} (\Gamma_l)=1 \}$. 
%   and $\ell = \max (I \setminus M_1)$.   
  Recall the partition ${\mathcal P}'$ of Definition \ref{def:partition}.  We will use the estimates in Lemma \ref{lem:theta} 
for the sets of ${\mathcal P}'$. 
  As a consequence, we obtain
\begin{equation}\label{eq:theta-initial-first-part}
   \Theta_{I \setminus \Omega_1} \geq 0, \;\mathrm{and}\;    \Theta_{I \setminus \Omega_1}  \geq 1
   \;\;\mathrm{if}\;\; F_{1 +  \max (I \setminus \Omega_1) } \neq F_{1 +  \max (I \setminus \Omega_1)}'.
\end{equation}   
At this point, there are two cases to consider.
\subsubsection{Case $1$:   $\Omega_1=\emptyset$}   We have $\Theta_{I}=  \Theta_{I \setminus \Omega_1}$.  It suffices to show that 
\[   \Theta_{I} + Z_{P_k} ({\mathcal F}_k, \gamma_k)  - \frac{1}{2} Z_{P_k} ({\mathcal H}_k, \gamma_k)  \geq 0. \]  
Suppose $F_{k} \neq F_{k}'$, i.e. $D_k$ is non-invariant.  Since
$\Theta_{I} \geq 1$, by Equation  \eqref{eq:theta-initial-first-part}, and $ Z_{P_k} ({\mathcal H}_k,  \gamma_k) = 1$ and 
$Z_{P_k} ({\mathcal F}_k, \gamma_k) \geq 0$ hold, 
the inequality follows. Assume now that  $F_{k} = F_{k}'$.  Then, the inequality is a consequence of 
$ Z_{P_k} ({\mathcal F}_k , \gamma_k)  \geq Z_{P_k} ({\mathcal H}_k,  \gamma_k) = 1$   (Lemma \ref{lem:fin}) and  $\Theta_{I} \geq 0$.

\subsubsection{Case 2: $\Omega_1 \neq \emptyset$}\label{case:omega-1}
We have $\Omega_1 = I_1$ and  $D_{\iota+1}$ is non-invariant (Definition \ref{def:i}).  
The point $P_{\iota}$ is a precursor by definition.
Since the strict transform of $D_{\iota +1}$ contains $P_j$ for any $\iota < j <k$, it follows that no point
$P_{\ell}$ with $\ell > \iota$ and   $\ell \in I \setminus \Omega_1$   is a precursor. 
In particular $S:=\{ \iota, \hdots, \rho \}$ belongs to ${\mathcal P}$.

Consider the  set $S'$ of ${\mathcal P}'$ containing $S$. 
By Lemma \ref{lem:theta}, we know that $\Theta_{S' \setminus S} \geq  \nu^{\gamma}_{\iota}$
and hence $\Theta_{I \setminus (\Omega_1  \cup S)} \geq  \nu^{\gamma}_{\iota}$  if $S \neq S'$.
Moreover, we obtain $\Theta_{I \setminus (\Omega_1  \cup S)} \geq 0$ if $S=S'$ again by Lemma \ref{lem:theta}. 
Notice that 
$\tau_{\iota} =  \nu_{P_{\iota}} ({\mathcal F}_{\iota})$ since $D_{\iota+1}$ is non-invariant. 
The previous discussion implies
\[ \Theta_{I \setminus (\Omega_1  \cup S)} + \nu^{\gamma}_{\iota} \theta_{\iota}  \geq  \nu^{\gamma}_{\iota} \]
by Lemma \ref{lem:half}.   Therefore, we only need to show 
\[
\frac{   \nu^{\gamma}_{\iota} + \sum_{l=\iota +1}^{\rho} \tau_{l} +
Z_{P_{\rho+1}}({\mathcal F}_{\rho+1}, \gamma_{\rho+1})}
{    \sum_{l=\iota + 1}^{k-1}
 (\nu_{P_l} ({\mathcal H}_l)-1) + Z_{P_k} ({\mathcal H}_k,  \gamma_k) }   \geq \frac{1}{2} .
\]  
By Lemma \ref{lem:half}, the following inequality
\[  \nu^{\gamma}_{\iota}  + \sum_{l=\iota +1}^{\rho} \tau_{l} \geq
\frac{1}{2} \left(    \sum_{l=\iota +1}^{\rho} 
(\nu_{P_l} ({\mathcal H}_l)-1) \right) + \nu_{\iota}^{\gamma} - (\rho - \iota)  \] 
holds, so that  as  $Z_{P_{\rho +1}} ({\mathcal F}_{\rho +1},  \gamma_{\rho +1})  \geq 0$, it suffices to show 
\begin{equation}
\label{equ:aux_mix}
\frac{   \nu_{\iota}^{\gamma} - (\rho - \iota) }
{  \sum_{l \in I_1}  
(\nu_{P_l} ({\mathcal H}_l)-1) + Z_{P_k} ({\mathcal H}_k,  \gamma_k) } \geq \frac{1}{2}. 
\end{equation}
 The denominator is at most
$1 + \sharp I_1$ by Lemma \ref{lem:fin} and Remark \ref{rem:mh}.
Since $ \nu_{\iota}^{\gamma}  \geq (\rho - \iota) + \sharp I_1$ and $\sharp I_1 \geq 1$, 
Equation (\ref{equ:aux_mix}) is a consequence of $2 \sharp I_1 \geq  \sharp I_1 + 1$.
This completes the proof of Theorem  \ref{teo:half}.

\end{document}